\documentclass[reqno]{siamart1116}
\usepackage[margin=1in]{geometry}
\usepackage{amsfonts}
\usepackage{amssymb}
\usepackage{amsmath}
\usepackage{amsopn}
\usepackage{stmaryrd}
\usepackage{graphicx}
\usepackage{xparse}
\usepackage{mathrsfs}
\usepackage{multirow}

\usepackage{subfigure}
\usepackage{bm}

\NewDocumentCommand{\dgal}{sO{}m}{%
  \IfBooleanTF{#1}
    {\dgalext{#3}}
    {\dgalx[#2]{#3}}%
}
\NewDocumentCommand{\dgalext}{m}{%
  \sbox0{%
    \mathsurround=0pt 
    $\left\{\vphantom{#1}\right.\kern-\nulldelimiterspace$%
  }%
  \sbox2{\{}%
  \ifdim\ht0=\ht2
    \{\kern-.45\wd2 \{#1\}\kern-.45\wd2 \}%
  \else
    \left\{\kern-.5\wd0\left\{#1\right\}\kern-.5\wd0\right\}%
  \fi
}

\NewDocumentCommand{\dgalx}{om}{%
  \sbox0{\mathsurround=0pt$#1\{$}%
  \sbox2{\{}%
  \ifdim\ht0=\ht2
    \{\kern-.45\wd2 \{#2\}\kern-.45\wd2 \}%
  \else
    \mathopen{#1\{\kern-.5\wd0 #1\{}
    #2
    \mathclose{#1\}\kern-.5\wd0 #1\}}
  \fi
}

\definecolor{lightblue}{rgb}{0.22,0.45,0.70}

\newsiamremark{remark}{Remark}

\numberwithin{theorem}{section}
\numberwithin{equation}{section}
\numberwithin{table}{section}
\numberwithin{figure}{section}
\newcommand{\vertiii}[1]{{\left\vert\kern-0.25ex\left\vert\kern-0.25ex\left\vert #1 
    \right\vert\kern-0.25ex\right\vert\kern-0.25ex\right\vert}}
\newcommand{\TheTitle}{Conforming, nonconforming and DG
methods for the stationary generalized Burgers-Huxley equation} 
\newcommand{\TheShortTitle}{FEM for the generalized Burgers-Huxley equation} 
\newcommand{\TheAuthors}{Arbaz Khan, Manil T Mohan and Ricardo Ruiz-Baier}
\newcommand{\TheShortAuthors}{A. Khan, M. T. Mohan \& R. Ruiz-Baier}

\headers{\TheShortTitle}{\TheShortAuthors}

\ifpdf
\hypersetup{
  pdftitle={\TheTitle},
  pdfauthor={\TheAuthors}
}
\fi

\def\L{L}

\def\no{\nonumber}

\def\H{H}

\allowdisplaybreaks

\title{{\TheTitle}%
\thanks{Submitted to the editors \today.%
\funding{
AK has been supported by the Sponsored Research \& Industrial Consultancy (SRIC), Indian Institute of Technology Roorkee, India through the faculty initiation
grant MTD/FIG/100878;  
 MTM has been supported by the Department of Science and Technology (DST), India through the Innovation in Science Pursuit for Inspired Research (INSPIRE) Faculty Award IFA17-MA110; and 
RRB has been supported by the Monash Mathematics Research Fund S05802-3951284, 
and by the HPC-Europa3 Transnational Access programme through grant HPC175QA9K. 
 }}}

\author{
 Arbaz Khan\thanks{Corresponding author. 
 Department of Mathematics,   Indian Institute of Technology Roorkee (IITR), 
 Roorkee,  India- 247667  (\email{arbaz@ma.iitr.ac.in}).}
 \and
 Manil T. Mohan\thanks{
 Department of Mathematics,   Indian Institute of Technology Roorkee (IITR), 
 Roorkee,  India- 247667 (\email{maniltmohan@ma.iitr.ac.in}).}
 \and
 Ricardo Ruiz-Baier\thanks{
 School of Mathematics, Monash University, 
 9 Rainforest Walk, Melbourne, VIC 3800, Australia,
  (\email{ricardo.ruizbaier@monash.edu}).}
}

\begin{document}

\maketitle

\begin{abstract}
In this work we address the analysis of the stationary generalized Burgers-Huxley equation (a nonlinear elliptic problem with anomalous advection)  and propose conforming, nonconforming and discontinuous Galerkin finite element methods for its numerical approximation. The existence, uniqueness and regularity of weak solutions is discussed in detail using a Faedo-Galerkin approach and fixed-point theory, and a priori error estimates for all three types of numerical schemes are rigorously derived. A set of computational results are presented to show the efficacy of the proposed methods.
\end{abstract}

\begin{keywords}
A priori error analysis, Conforming finite element method, Non-conforming finite element, discontinuous Galerkin, 
Stationary generalized Burgers-Huxley equation.
\end{keywords}

\begin{AMS}
65N15,   	
65N30, 
	35J66, 
		65J15 
\end{AMS}

\section{Introduction}\label{sec1}\setcounter{equation}{0} 
The Burgers-Huxley equation is a special type of nonlinear advection-diffusion-reaction problems that are of importance in applications in mechanical engineering, material sciences, and neurophysiology. Some examples  include, for instance,  particle transport \cite{JS}, dynamics of ferroelectric materials \cite{OYNA}, action potential propagation in nerve fibers \cite{XYW}, wall motion in liquid crystals \cite{wang85}, and many others (see also \cite{VJE,MTAR} and the references therein).

Our starting point is the following stationary form of the generalized Burgers-Huxley equation with Dirichlet boundary conditions 
\begin{equation}\label{8.1}
\left\{
\begin{aligned}
-\nu\Delta u+\alpha u^{\delta}\sum\limits_{i=1}^d\frac{\partial u}{\partial x_i}-\beta u (1-u^{\delta} )(u^{\delta} -\gamma)&=f, \ \text{ in }\ \Omega, \\
u&=0,\ \text{ on }\ {\partial}\Omega,
\end{aligned}\right.
\end{equation}
where it is assumed that $\Omega\subset \mathbb{R}^d \ (d=2,3)$ is an open bounded and simply connected domain with Lipschitz boundary $\partial\Omega$. Here  $\nu>0$ is the constant diffusion coefficient, 
$\alpha>0$ is the advection coefficient, and $\beta>0$, $\delta\geq 1$, $\gamma\in(0,1)$ are model parameters modulating the interplay between non-standard nonlinear advection, diffusion, and nonlinear reaction (or applied current) contributions. 
 
  The global solvability of the one-dimensional Burgers-Huxley equation has been recently established  in \cite{MTAR}. In this paper we extend the analysis to the multi-dimensional   case.     Drawing inspiration from the techniques usually employed for the analysis of steady Navier-Stokes  equations (cf. \cite[Ch. 10]{Te}), we use a Faedo-Galerkin approximation, Brouwer's fixed-point theorem, and compactness arguments to derive the existence and uniqueness of weak solutions to the two- and three-dimensional stationary generalized Bur\-gers-Huxley equation in bounded domains with Lipschitz boundary and under a minimal regularity assumption. For the case of domains that are convex or have $C^2-$boundary, we employ the   elliptic regularity results available in, e.g., \cite{brezis_book11,grisvard_book85}, and establish that the weak solution of \eqref{8.1} satisfies $u\in H^2(\Omega)\cap H^1_0(\Omega)$. 
    
The recent literature relevant to the construction and analysis of discretizations for \eqref{8.1} and closely related problems is very diverse. For instance, numerical methods specifically designed to capture boundary layers in 
 singularly perturbed generalized Bur\-gers-Huxley equations have been studied in \cite{kumar11}, 
different types of finite differences have been used in \cite{sari11,macias18,shukla20,verma20}, spectral, B-spline and Chebyshev wavelet collocation methods 
 have been advanced in \cite{alinia19,javidi06,wasim18,celik16}, numerical solutions obtained with the so-called adomain decomposition were analyzed in \cite{hashim06}, homotopy perturbation techniques were used in 
 \cite{maurya19}, Strang splittings were proposed in \cite{cicek16}, meshless radial basis functions were studied in \cite{khattak09}, 
 generalized finite differences and finite volume schemes have been analyzed in \cite{chen03,zhou19} for the restriction of \eqref{8.1} to the diffusive Nagumo (or bistable) model, and  a finite element method satisfying a discrete maximum principle was introduced in \cite{VJE} (the latter reference is closer to the present study).  
Although there is a growing interest in developing numerical techniques for the generalized Bur\-gers-Huxley equation, it appears that 
the aspects of error analysis for finite element discretizations have not been yet thoroughly  addressed. Then, somewhat differently from the methods listed above (where we stress that such list is far from complete), here we propose a family of schemes consisting of conforming finite elements (CFEM),   non-conforming finite elements (NCFEM) and discontinuous Galerkin methods (DGFEM). Following the assumptions adopted for the continuous problem, we rigorously derive a priori error estimates indicating first-order convergence of the CFEM. In contrast, for NCFEM and DGFEM the solvability of the discrete problem does not follow from the continuous problem, but separate conditions are established to ensure the existence of discrete solutions in these cases. The minimal assumptions on the domain are also used to prove first-order a priori error bounds for NCFEM and DGFEM, and we briefly comment about $L^2-$estimates. We also include a set of computational tests that confirm the theoretical error bounds and which also show some properties of the model equation.

We have organized the remainder of the paper as follows: Section \ref{sec2} contains notational conventions and it presents the well-posedness and regularity analysis of (\ref{8.1}), discussing also some possible modifications to the proofs of existence and uniqueness of weak solutions. The numerical discretizations are introduced and then a priori error estimates are derived for CFEM, NCFEM and DGFEM in Section \ref{sec3}. Finally, Section \ref{sec4} has a compilation of numerical tests in 2D and 3D that serve to illustrate our theoretical results.

\section{Solvability of the stationary generalized Burgers-Huxley equation}\label{sec2}\setcounter{equation}{0} 
\subsection{Preliminaries} 
Throughout this section we will adopt the usual notation for functional spaces. 
In particular, for $p \in [1,\infty)$ we denote the Banach space of Lebesgue $p-$integrable functions by 
$$ L^{p}(\Omega) :=\left\{u: \int_{\Omega}|u(x)|^{p}d x <\infty \right\},$$
whereas for $p = \infty$, $L^{\infty}(\Omega)$ is the space conformed by essentially bounded measurable functions 
on the domain. Moreover, for integers $s\ge0$, by $H^s(\Omega)$ we denote the standard Sobolev spaces $W^{s,2}(\Omega)$, endowed with the norm
$\|u\|_{s,\Omega}^2 = \| u \|^2_{0,\Omega} + \sum_{|i|\leq s} \|\partial^i u\|^2_{0,\Omega}$. For $s=0$, we adopt the convention $H^0(\Omega)=L^2(\Omega)$, and recall the definition of the closure of all $C^\infty$ functions with 
compact support in $H^1(\Omega)$ 
$H^1_0(\Omega) :=\{u\in H^1(\Omega): u|_{\partial \Omega}=0\ \text{a.e.}\}$.  
If $Y(M)$ denotes a generic normed space of functions over the spatial domain $M$, then the associated norm will be 
at some instances denoted as $\|\cdot\|_{Y}$ (omitting the domain specification whenever clear from the context). 
In addition, let $H^{-1}(\Omega)$ be the dual space of the Sobolev space $H^1_0(\Omega)$ with the following norm
\begin{align*}
\|u\|_{H^{-1}(\Omega)}:=\sup_{0\neq v\in H^{1}_0(\Omega)}\frac{\langle u,v\rangle}{\|v\|_{1,\Omega}},
\end{align*} 
where $\langle\cdot,\cdot\rangle$ denotes the duality pairing between $H_0^1(\Omega)$ and $H^{-1}(\Omega)$. In the sequel, we use the same notation for the duality pairing between $L^p(\Omega)$ and its dual $L^{\frac{p}{p-1}}(\Omega)$, for $p\in(2,\infty)$. 

We proceed to rewrite problem (\ref{8.1}) in the following abstract form: 
\begin{align}\label{8p2}
\nu Au +\alpha B(u )-\beta C(u )=f,
\end{align}
where the involved operators are 
\begin{align*}
Au =-\Delta u, \quad B(u)= u^{\delta}\sum\limits_{i=1}^d\frac{\partial u}{\partial x_i},\quad \text{ and }\ C(u)=u (1-u^{\delta} )(u^{\delta} -\gamma). 
\end{align*}
For the Dirichlet Laplacian operator $A$, it is well-known that  $D(A)=H^2(\Omega)\cap H_0^1(\Omega)\subset L^p$, for $p\in[1,\infty)$ and $1\leq d\leq 4$, using the Sobolev Embedding Theorem (see, e.g., \cite{grisvard_book85}) and also $A:H_0^1(\Omega)\to H^{-1}(\Omega)$. Since $\Omega$ is bounded, the embedding  $H_0^1(\Omega)\subset L^2(\Omega)$ is compact, and hence using the spectral theorem, there exists a sequence $0<\lambda_1\leq\lambda_2\leq\ldots\to\infty$ of eigenvalues of $A$ and an orthonormal basis $\{w_k\}_{k=1}^{\infty}$ of $L^2(\Omega)$ consisting of eigenfunctions of $A$ \cite[p. 504]{RDJL}. Furthermore, we have the following Friedrichs-Poincar\'e inequality: $\sqrt{\lambda_1}\|u\|_0\leq \|\nabla u\|_0$. 

Testing \eqref{8.1}  against a smooth function $v$, integrating by parts, and applying the boundary 
condition, we end up with the following problem in weak form: 
Given any $f\in\H^{-1}(\Omega)$, find $u \in H_0^1(\Omega)$ such that 
\begin{align}\label{8p3}
\nu(\nabla u ,\nabla v)+\alpha b(u,u,v) -\beta\langle C(u ),v\rangle=\langle f,v\rangle, \quad \text{ for all }\ v\in\H_0^1(\Omega),
\end{align}
where $b(u,u,v)=\langle B(u ),v\rangle$. 

\subsection{Existence of weak solutions}
Let us first address the well-posedness of \eqref{8.1} in two dimensions.
 
\begin{theorem}[Existence of weak solutions]\label{thm6.1}
	For a given $f\in\H^{-1}(\Omega)$, there exists at least one solution to the Dirichlet problem \eqref{8.1}.
\end{theorem}
\begin{proof}
We prove the existence result using the following steps.
\vskip 0.2 cm
\noindent 
 \textbf{Step 1: Finite dimensional system.} We formulate a Faedo-Galerkin approximation method. Let the functions $w_k=w_k(x),$ $k=1,2,\ldots,$  be smooth, the set $\{w_k(x)\}_{k=1}^{\infty}$ be an orthogonal basis of  $\H_0^1(\Omega)$ and orthonormal basis of $\L^2(\Omega)$. One can take $\{w_k(x)\}_{k=1}^{\infty}$ as the complete set of normalized eigenfunctions of the operator $-\Delta$ in $\H_0^1(\Omega)$. 
	For a fixed positive integer $m$, we look for a function $u_m\in\H_0^1(\Omega)$ of the form 
	\begin{equation}\label{8p4}
	u_m=\sum\limits_{k=1}^m\xi_m^kw_k,\ \xi_m^k\in\mathbb{R},
	\end{equation} 
	and 
	\begin{equation}\label{8p5}
	\nu(\nabla u_m,\nabla w_k)+\alpha b(u_m,u_m,w_k)-\beta\langle C(u_m),w_k\rangle=\langle f,w_k\rangle,
	\end{equation}
	for   $k=1,\ldots,m$. The set of equations in \eqref{8p5} is equivalent to 
	\[
	\nu Au_m+\alpha P_mB(u_m)-\beta P_mc(u_m)=P_mf. 
	\]
	Equations \eqref{8p4}-\eqref{8p5} constitute a nonlinear system for $\xi_m^1,\ldots,\xi_m^m$. We invoke \cite[Lem. 1.4]{Te} (an application of Brouwer's fixed point theorem) to prove the existence of solution to such a system. Let us consider the space $W=\text{Span}\left\{w_1,\ldots,w_m\right\}$ and the associated scalar product  $[\cdot,\cdot]=(\nabla\cdot,\nabla\cdot)$. We define  the map $P=P_m$ as 
	\[
	[P_m(u),v]=(\nabla P_m(u),\nabla v)=\nu(\nabla u,\nabla v)+\alpha b(u,u,v)-\beta \langle C(u),v\rangle -\langle f,v\rangle,
	\]
	for all $u,v\in W$. The continuity of  $P_m$ can be verified in the following way 
	\begin{align*}
&	|[P_m(u),v]|\\&\leq\left(\nu\|\nabla u\|_0+\frac{\alpha}{\delta+1}\|u\|_{L^{2(\delta+1)}}^{\delta+1}\right)\|\nabla v\|_0+\beta\left[(1+\gamma)\|u\|_{L^{2(\delta+1)}}^{\delta+1}+\gamma\|u\|_{0}\right]\|v\|_0\\&\quad+\|u\|_{\L^{2(\delta+1)}}^{2\delta+1}\|v\|_{L^{2\delta+1}}\nonumber\\&\leq \left[\left(\nu+\frac{\beta\gamma}{\lambda_1^2}\right)\|\nabla u\|_0+\left(\frac{\alpha}{\delta+1}+\frac{\beta(1+\gamma)}{\lambda_1}\right)\|u\|_{L^{2(\delta+1)}}^{\delta+1}+\|u\|_{\L^{2(\delta+1)}}^{2\delta+1}\right]\|\nabla v\|_0,
	\end{align*}
	for all $v\in H_0^1(\Omega)$. Using Sobolev's embedding, we know that $H_0^1(\Omega)\subset L^p(\Omega)$, for all $p\in[2,\infty)$, and hence the continuity follows. In order to apply \cite[Lem. 1.4]{Te}, we need to show that $$[P_m(u),u]>0, \ \text{ for } \ [u]=k>0,$$ where $[\cdot]$ denotes the norm on $W$, which is in turn the norm induced by $\H_0^1(\Omega)$. We can then use Poincar\'e's, H\"older's and Young's inequalities to estimate $[P_m(u),u]$ as 
	\begin{align*}
&	[P_m(u),u]\nonumber\\&=\nu\|\nabla u\|_{0}^2+\beta\gamma\|u\|_{0}^2+\beta\|u\|_{\L^{2\delta+2}}^{2\delta+2}-\beta(1+\gamma)(u^{\delta+1},u)-(f,u)\nonumber\\
	&\geq \nu\|\nabla u\|_{0}^2+\beta\gamma\|u\|_{0}^2+\beta\|u\|_{\L^{2\delta+2}}^{2\delta+2}-\beta(1+\gamma)\|u\|_{\L^{2\delta+2}}^{\delta+1}\|u\|_{0}-\|f\|_{\H^{-1}}\|\nabla u\|_{0}\nonumber\\
	&\geq \frac{\nu}{2}\|\nabla u\|_{0}^2+\beta\gamma\|u\|_{0}^2+\frac{\beta}{2}\|u\|_{\L^{2\delta+2}}^{2\delta+2} -\frac{\beta\delta(1+\gamma)^{\frac{2(\delta+1)}{\delta}}}{2(\delta+1)}\left(\frac{\delta+2}{\delta+1}\right)^{\frac{\delta+2}{\delta}}|\Omega|-\frac{1}{2\nu}\|f\|_{\H^{-1}}^2 \nonumber\\
	&\geq\frac{\nu}{2}\|\nabla u\|_{0}^2-\frac{\beta\delta(1+\gamma)^{\frac{2(\delta+1)}{\delta}}}{2(\delta+1)}\left(\frac{\delta+2}{\delta+1}\right)^{\frac{\delta+2}{\delta}}|\Omega|-\frac{1}{2\nu}\|f\|_{\H^{-1}}^2,
	\end{align*}
	where $|\Omega|$ is the Lebesgue measure of $\Omega$. 
	It follows that $[P_m(u),u]$ $> 0,$ for $\|u\|_{1}=\kappa,$ where $\kappa$ is sufficiently large. More precisely, the 
	analysis requires $$\kappa>\sqrt{\frac{2}{\nu}\left(\frac{\beta\delta(1+\gamma)^{\frac{2(\delta+1)}{\delta}}}{2(\delta+1)}\left(\frac{\delta+2}{\delta+1}\right)^{\frac{\delta+2}{\delta}}|\Omega|+\frac{1}{2\nu}\|f\|_{\H^{-1}}^2\right)}. $$ Thus the hypotheses of \cite[Lem. 1.4]{Te} are satisfied and a solution $u_m$ to \eqref{8p5} exists. 
	
	\bigskip 
	
 \noindent\textbf{Step 2: Uniform boundedness.} 	Next we need to show that  the solution $u_m$ is bounded. Multiplying \eqref{8p5} by $\xi_m^k$ and then adding from $k=1,\ldots,m$, we find 
	\begin{align}\label{8.9}
	&\nu\|\nabla u_m\|_{0}^2+\beta\|u_m\|_{\L^{2\delta+2}}^{2\delta+2}+\beta\gamma\|u_m\|_{0}^2\nonumber\\&=\beta(1+\gamma)(u_m^{\delta+1},u_m)+\langle f,u_m\rangle \nonumber\\&\leq\beta(1+\gamma)\|u_m\|_{\L^{2\delta+2}}^{\delta+2}|\Omega|^{\frac{\delta}{2(\delta+1)}}+\|f\|_{\H^{-1}}\|u_m\|_{1}\nonumber\\&\leq\frac{\beta}{2}\|u_m\|_{\L^{2\delta+2}}^{2\delta+2}+\frac{\beta\delta(1+\gamma)^{\frac{2(\delta+1)}{\delta}}}{2(\delta+1)}\left(\frac{\delta+2}{\delta+1}\right)^{\frac{\delta+2}{\delta}}|\Omega|+\frac{\nu}{2}\|u_m\|_{1}^2+\frac{1}{2\nu}\|f\|_{\H^{-1}}^2,
	\end{align}
	where we have used H\"older's and Young's inequalities. From \eqref{8.9}, we deduce that 
	\begin{align}\label{2p13}
	\nu\|u_m\|_{1}^2+\beta\|u_m\|_{\L^{2\delta+2}}^{2\delta+2}\leq\frac{\beta\delta(1+\gamma)^{\frac{2(\delta+1)}{\delta}}}{\delta+1}\left(\frac{\delta+2}{\delta+1}\right)^{\frac{\delta+2}{\delta}}|\Omega|+\frac{1}{\nu}\|f\|_{\H^{-1}}^2.
	\end{align}
	
		\bigskip
	
\noindent\textbf{Step 3: Passing to the limit.}  
	We have bounds for $\|u_m\|_{1}^2$ and $\|u_m\|_{\L^{2\delta+2}}^{2\delta+2}$ that are uniform and independent of $m$. Since $\H_0^1(\Omega)$ and  $\L^{2\delta+2}(\Omega)$ are reflexive, using the Banach-Alaoglu theorem, we can extract a subsequence $\{u_{m_k}\}$ of $\{u_m\}$ such that 
	\[
	\begin{cases}
	u_{m_k}&\xrightarrow{w} u , \ \text{ in }\ \H_0^1(\Omega), \ \text{ as }\ k\to\infty, \\
		u_{m_k}&\xrightarrow{w} u , \ \text{ in }\ \L^{2\delta+2}(\Omega), \ \text{ as }\ k\to\infty. 
	\end{cases}
	\]
In two dimensions we have that $\H_0^1(\Omega)\subset\L^{2\delta+2}(\Omega)$, thanks to the Sobolev embedding theorem. 	Since the embedding of $\H_0^1(\Omega)\subset\L^2(\Omega)$ is compact, one can extract a subsequence $\{u_{m_{k_j}}\}$ of $\{u_{m_k}\}$ such that 
	\begin{align}\label{2p15}
	u_{m_{k_j}}\to u , \ \text{ in }\ \L^2(\Omega), \ \text{ as }\ j\to\infty. 
	\end{align}
	Passing to limit in \eqref{8p5} along the subsequence $\{m_{k_j}\}$, we find that $u $ is a solution to \eqref{8p3}, provided one can show that 
	\begin{align*}
	B(u_{m_{k_j}})\xrightarrow{w} B(u) , \ \text{ and }\ 	C(u_{m_{k_j}})\xrightarrow{w} C(u )\ \text{ in }\ \H^{-1}(\Omega), \ \text{ as } \ j\to\infty. 
	\end{align*}
	In order to do this, we first show that $b(u_{m_{k_j}},u_{m_{k_j}},v)\to b(u,u,v),$ for all $v\in C_0^{\infty}(\Omega)$. Then, using a density argument, we obtain that $	B(u_{m_{k_j}})\xrightarrow{w} B(u) \ \text{ in }\ \H^{-1}(\Omega)$, as  $j\to\infty$. Using an integration by parts, Taylor's formula \cite[Th. 7.9.1]{ciarlet_book13}, H\"older's inequality, the estimate \eqref{2p13}, and convergence \eqref{2p15}, we obtain 
	\begin{align}\label{215}
&	|b(u_{m_{k_j}},u_{m_{k_j}},v)- b(u,u,v)|\nonumber\\&= \left|\frac{1}{\delta+1}\sum_{i=1}^2\int_{\Omega}(u_{m_{k_j}}^{\delta+1}(x)-u^{\delta+1}(x))\frac{\partial v(x)}{\partial x_i}d x\right|\nonumber\\&=\left|\sum_{i=1}^2\int_{\Omega}(\theta u_{m_{k_j}}(x)+(1-\theta)u(x))^{\delta}(u_{m_{k_j}}(x)-u(x))\frac{\partial v(x)}{\partial x_i}d x\right|\nonumber\\&\leq\|u_{m_{k_j}}-u\|_{0}\left(\|u_{m_{k_j}}\|_{\L^{2(\delta+1)}}^{\delta}+\|u\|_{\L^{2(\delta+1)}}^{\delta}\right)\|\nabla v\|_{\L^{2(\delta+1)}}\nonumber\\&\to 0\ \text{ as } \ j\to\infty, \ \text{ for all } \ v\in C_0^{\infty}(\Omega). 
	\end{align}
	Making use again of Taylor's formula, interpolation and H\"older's inequalities, we find 
	\begin{align}\label{216}
&	|(C(u_{m_{k_j}})-C(u),v)|\nonumber\\&\leq(1+\gamma)\left|\int_{\Omega}(u_{m_{k_j}}^{\delta+1}(x)-u^{\delta+1}(x))v(x)d x\right|+\left|\int_{\Omega}(u_{m_{k_j}}(x)-u(x))v(x)d x\right| \nonumber\\
&\quad+\left|\int_{\Omega}(u_{m_{k_j}}^{2\delta+1}(x)-u^{2\delta+1}(x))v(x)d x\right|\nonumber\\&\leq (1+\gamma)(\delta+1)\int_{\Omega}\left|(u_{m_{k_j}}(x)-u(x))(\theta u_{m_{k_j}}(x)+(1-\theta)u(x))^{\delta}v(x)\right|d x\nonumber\\
&\quad+\int_{\Omega}\left|(u_{m_{k_j}}(x)-u(x))v(x)\right|d x\nonumber\\
&\quad+(1+2\delta)\int_{\Omega}\left|(u_{m_{k_j}}(x)-u(x))(\theta u_{m_{k_j}}(x)+(1-\theta)u(x))^{2\delta}v(x)\right|d x\nonumber\\
&\leq(1+\gamma)(\delta+1)\|u_{m_{k_j}}-u\|_{0}\left(\|u_{m_{k_j}}\|_{\L^{2(\delta+1)}}^{\delta}+\|u\|_{\L^{2(\delta+1)}}^{\delta}\right)\|v\|_{\L^{2(\delta+1)}}\nonumber\\
&\quad+\|u_{m_{k_j}}-u\|_{0}\|v\|_{0}
+(1+2\delta)\|u_{m_{k_j}}-u\|_{\L^{\delta+1}}\left(\|u_{m_{k_j}}\|_{\L^{2(\delta+1)}}^{2\delta}+\|u\|_{\L^{2(\delta+1)}}^{2\delta}\right)\|v\|_{\L^{\infty}}\nonumber\\
&\leq \left((1+\gamma)(\delta+1)
\left(\|u_{m_{k_j}}\|_{\L^{2(\delta+1)}}^{\delta}+\|u\|_{\L^{2(\delta+1)}}^{\delta}\right)\|v\|_{\L^{2(\delta+1)}}+\|v\|_0\right)\|u_{m_{k_j}}-u\|_{0}\nonumber\\
&\quad+(1+2\delta)\|u_{m_{k_j}}-u\|_{0}^{\frac{1}{\delta}}\left(\|u_{m_{k_j}}\|_{\L^{2(\delta+1)}}^{1-\frac{1}{\delta}}+\|u\|_{\L^{2(\delta+1)}}^{1-\frac{1}{\delta}}\right)\times\nonumber\\
&\quad\quad\left(\|u_{m_{k_j}}\|_{\L^{2(\delta+1)}}^{2\delta}+\|u\|_{\L^{2(\delta+1)}}^{2\delta}\right)\|v\|_{\L^{\infty}}\to 0\ \text{ as } \ j\to\infty, \ \text{ for all } \ v\in C_0^{\infty}(\Omega).
	\end{align}
	Moreover, $u $ satisfies \eqref{8p3} and 
	\begin{align}\label{8.14}
	\nu\|u\|_{1}^2+\beta\|u\|_{\L^{2\delta+2}}^{2\delta+2}\leq\frac{\beta\delta(1+\gamma)^{\frac{2(\delta+1)}{\delta}}}{\delta+1}\left(\frac{\delta+2}{\delta+1}\right)^{\frac{\delta+2}{\delta}}|\Omega|+\frac{1}{\nu}\|f\|_{\H^{-1}}^2=:\widetilde{K},
	\end{align}
which completes the existence proof. 
\end{proof}

\subsection{Uniqueness of weak solution}
\begin{theorem}[Uniqueness]\label{thm6.2}
	Let $f\in\H^{-1}(\Omega)$ be given. Then, for  
		\begin{align}\label{8pp5}
		\nu>\max\left\{\frac{4^{\delta}\alpha^2}{\beta},\frac{\beta}{\lambda_1}\left[4^{\delta}(1+\gamma)^2(1+\delta)^2-{2\gamma}\right]\right\},
		\end{align}
		where $\lambda_1$ is the first eigenvalue of the Dirichlet Laplacian operator, the solution of \eqref{8p3} is unique. 
\end{theorem}
\begin{proof} 
We assume $u $ and $v$ are two weak solutions of \eqref{8p3} and define $w:=u -v$. Then $w$ satisfies:
	\begin{align}\label{8p17}
	\nu(\nabla w,\nabla v)+\alpha\langle B(u )-B(v),v\rangle -\beta\langle C(u )-C(v),v\rangle =0,
	\end{align}
	for all $v\in\H_0^1(\Omega)$. Taking $v=w$ in \eqref{8p17}, we have 
	\begin{align}\label{8p18}
	\nu\|\nabla w\|_{0}^2&=-\alpha\langle B(u)-B(v),w\rangle+\beta \langle C(u )-C(v),w\rangle. 
	\end{align}
		Then it can be readily seen that
	\begin{align}\label{3.24}
	&\beta\left[\langle  u(1-u^{\delta})(u^{\delta}-\gamma)-v(1-v^{\delta})(v^{\delta}-\gamma),w\rangle \right]\nonumber\\&=-\beta\gamma\|w\|_{0}^2-\beta(u^{2\delta+1}-v^{2\delta+1},w)+\beta(1+\gamma)(u^{\delta+1}-v^{\delta+1},w).
	\end{align}
	Let us take the term $-\beta(u^{2\delta+1}-v^{2\delta+1},w)$ from \eqref{3.24} and estimate it using H\"older's and Young's inequalities as 
	\begin{align}\label{3.25}
	-\beta(u^{2\delta+1}-v^{2\delta+1},w)&=-\beta(|u|^{2\delta}(u-v)+|u|^{2\delta}v-|v|^{2\delta}u,w+|v|^{2\delta}(u-v),w)\nonumber\\
	&=-\beta\|u^{\delta}w\|_{0}^2-\beta\|v^{\delta}w\|_0^2-\beta(|u|^{2\delta}+|v|^{2\delta},uv)+\beta(|u|^2,|v|^{2\delta})\nonumber\\&\quad+\beta(|v|^2,|u|^{2\delta})\nonumber\\&= -\frac{\beta}{2}\|u^{\delta}w\|_{0}^2-\frac{\beta}{2}\|v^{\delta}w\|_0^2-\frac{\beta}{2}((|u|^{2\delta}-|v|^{2\delta}),(|u|^2-|v|^2))\nonumber\\&\leq -\frac{\beta}{2}\|u^{\delta}w\|_{0}^2-\frac{\beta}{2}\|v^{\delta}w\|_0^2.
	\end{align}
	Next, we take the term $\beta(1+\gamma)(u^{\delta+1}-v^{\delta+1},w)$ from \eqref{3.24} and estimate it using Taylor's formula, H\"older's and Young's inequalities as 
	\begin{align}\label{3.26}
	& \beta(1+\gamma)(u^{\delta+1}-v^{\delta+1},w)\nonumber\\&=\beta(1+\gamma)(\delta+1)((\theta  u+(1-\theta)v)^{\delta}w,w)\nonumber\\&\leq \beta(1+\gamma)(\delta+1)2^{\delta-1}(\|u^{\delta}w\|_{0}+\|v^{\delta}w\|_{0})\|w\|_{0}\nonumber\\&\leq\frac{\beta}{4}\|u^{\delta}w\|_{0}^2+\frac{\beta}{4}\|v^{\delta}w\|_{0}^2+\frac{\beta}{2}2^{2\delta}(1+\gamma)^2(\delta+1)^2\|w\|_{0}^2.
	\end{align}
	Combining \eqref{3.25}-\eqref{3.26} and substituting the result back into \eqref{3.24}, we obtain 
	\begin{align}\label{327}
	&\beta\left[(u(1-u^{\delta})(u^{\delta}-\gamma)-v(1-v^{\delta})(v^{\delta}-\gamma),w)\right]\nonumber\\
	&\leq -\beta\gamma\|w\|_{0}^2-\frac{\beta}{4}\|u^{\delta}w\|_{0}^2-\frac{\beta}{4}\|v^{\delta}w\|_{0}^2+\frac{\beta}{2}2^{2\delta}(1+\gamma)^2(\delta+1)^2\|w\|_{0}^2.
	\end{align}
	On the other hand, we derive a bound for $-\alpha\langle B(u)-B(v),w\rangle$ using an integration by parts, Taylor's formula,  H\"older's and Young's inequalities. This gives  
	\begin{align}\label{6.49}
	-\alpha\langle B(u)-B(v),w\rangle&=\frac{\alpha}{\delta+1} \left((u^{\delta+1}-v^{\delta+1})\left(\begin{array}{c}1\\1\end{array}\right),\nabla w\right)\nonumber\\&=\alpha \left((u-v)(\theta u+(1-\theta)v)^{\delta}\left(\begin{array}{c}1\\1\end{array}\right),\nabla w\right)\nonumber\\&\leq 2^{\delta-1}\alpha\|\nabla w\|_{0}\left(\|u^{\delta} w\|_{0}+\|v^{\delta} w\|_{0}\right)\nonumber\\
	&\leq\frac{\nu}{2}\|\nabla w\|_{0}^2+\frac{2^{2\delta}\alpha^2}{4\nu}\|u^{\delta} w\|_{0}^2+\frac{2^{2\delta}\alpha^2}{4\nu}\|v^{\delta} w\|_{0}^2.	\end{align}
	Combining \eqref{327}-\eqref{6.49},  and substituting that back in \eqref{8p18}, we further have 
	\begin{align}\label{8.19}
&\left[	\frac{ \nu}{2}+\frac{1}{\lambda_1}\left(\beta\gamma-\frac{\beta}{2}2^{2\delta}(1+\gamma)^2(\delta+1)^2\right)\right]\|\nabla w\|_{0}^2\nonumber\\&\quad+\left(\frac{\beta}{4}-\frac{2^{2\delta}\alpha^2}{4\nu}\right)\|u^{\delta} w\|_{0}^2+\left(\frac{\beta}{4}-\frac{2^{2\delta}\alpha^2}{4\nu}\right)\|v^{\delta} w\|_{0}^2\leq 0.
	\end{align}
	It should also be noted that 
	\begin{align*}
	\|u-v\|_{L^{2\delta+2}}^{2\delta+2}&=\int_{\Omega}|u(x)-v(x)|^{2\delta}|u(x)-v(x)|^2d x\nonumber\\&\leq 2^{2\delta-1}(\|u^{\delta}(u-v)\|_{0}^2+\|v^{\delta}(u-v)\|_{0}^2).
	\end{align*}
	Thus from \eqref{8.19}, it is immediate to see that 
		\begin{align*}
	&\left[	\frac{ \nu}{2}+\frac{1}{\lambda_1}\left(\beta\gamma-\frac{\beta}{2}4^{\delta}(1+\gamma)^2(\delta+1)^2\right)\right]\|\nabla w\|_{0}^2+\frac{1}{2^{2\delta+1}}\left(\beta-\frac{4^{\delta}\alpha^2}{\nu}\right)\|w\|_{L^{2\delta+2}}^{2\delta+2}\leq 0,
	\end{align*}
and for the condition given in \eqref{8.19},  the uniqueness readily follows. 
\end{proof}

\subsection{Possible modifications in the proofs, and a regularity result}

\begin{remark}
If one uses Gagliardo-Nirenberg interpolation inequality	to  estimate the term $-\alpha\langle B(u)-B(v),w\rangle$, then it can be easily seen that 
\begin{align}\label{2.31}
-\alpha\langle B(u)-B(v),w\rangle&\leq \alpha\|\nabla w\|_{0}\|w\|_{\L^{2(\delta+1)}}\left(\|u\|_{\L^{2(\delta+1)}}^{\delta}+\|v\|_{\L^{2(\delta+1)}}^{\delta}\right)\nonumber\\
&\leq C\alpha\|\nabla w\|_{0}^{\frac{2\delta+1}{\delta+1}}\left(\|u\|_{\L^{2(\delta+1)}}^{\delta}+\|v\|_{\L^{2(\delta+1)}}^{\delta}\right)\|w\|_{0}^{\frac{1}{\delta+1}}\nonumber\\
&\leq \frac{C\alpha}{\lambda_1^{\frac{1}{2(\delta+1)}}}\left(\|u\|_{\L^{2(\delta+1)}}^{\delta}+\|v\|_{\L^{2(\delta+1)}}^{\delta}\right)\|\nabla w\|_{0}^2\nonumber\\
&\leq\frac{2C\alpha}{\lambda_1^{\frac{1}{2(\delta+1)}}}\sqrt{\frac{\widetilde{K}}{\beta}}\|\nabla w\|_{0}^2,
\end{align}
where $C$ is the constant appearing in the Gagliardo-Nirenberg  inequality. Combining \eqref{327} and \eqref{2.31}, and substituting it in \eqref{8p18}, we get 
\begin{align*}
&\left[\nu+\frac{1}{\lambda_1}\left(\beta\gamma-\frac{\beta}{2}2^{2\delta}(1+\gamma)^2(\delta+1)^2\right)-\frac{2C\alpha}{\lambda_1^{\frac{1}{2(\delta+1)}}}\sqrt{\frac{\widetilde{K}}{\beta}}\right]\|\nabla w\|_{0}^2\leq 0,
\end{align*}
Thus the uniqueness follows provided 
\begin{align}\label{234}
\nu+\frac{\beta\gamma}{\lambda_1}>\frac{\beta}{\lambda_1}2^{2\delta-1}(1+\gamma)^2(\delta+1)^2+\frac{2C\alpha}{\lambda_1^{\frac{1}{2(\delta+1)}}}\sqrt{\frac{\widetilde{K}}{\beta}},
\end{align}
where $\widetilde{K}$ is defined in \eqref{8.14}.
\end{remark}
	
\begin{remark}
For $\delta=1$ (that is, for the classical Burgers-Huxley equation), we obtain a simpler condition than \eqref{8pp5} for the uniqueness of weak solution. In this case, the estimate \eqref{327} becomes (see \cite{MTAR})
		\begin{align}\label{231}
	&\beta\left[(u(1-u)(u-\gamma)-v(1-v)(v-\gamma),w)\right]\nonumber\\
	&\leq -\beta\|uw\|_{0}^2-\beta\|vw\|_{0}^2+\beta(1+\gamma+\gamma^2)\|w\|_{0}^2.
	\end{align}
	Similarly, we estimate the term $-\alpha\langle B(u)-B(v),w\rangle$  as 
	\begin{align}\label{235}
	-\alpha\langle B(u)-B(v),w\rangle&=-\alpha [b(w,w,w) +b(w,v,w)+b(v,w,w)]\nonumber\\&=\alpha b(v,w,w)\leq
	\frac{\nu}{2}\|\nabla w\|_{0}^2+\frac{\alpha^2}{2\nu}\|vw\|_{0}^2. 
	\end{align}
	Thus, as an immediate consequence we have that 
	\[
	\left[\frac{\nu}{2}-\frac{\beta(1+\gamma+\gamma^2)}{\lambda_1}\right]\|\nabla w\|_{0}^2+ \beta\|uw\|_{0}^2+\left(\beta-\frac{\alpha^2}{2\nu}\right)\|uw\|_{0}^2\leq 0,
	\]
	and hence for $$\nu>\max\left\{\frac{2\beta(1+\gamma+\gamma^2)}{\lambda_1},\frac{\alpha^2}{2\beta}\right\},$$ the uniqueness of weak solution holds. To conclude, one can use the Ladyzhenskaya inequality to estimate $-\alpha\langle B(u)-B(v),w\rangle$. Then, 
	the bound  \eqref{235} becomes
	\begin{align}\label{237}
		-\alpha\langle B(u)-B(v),w\rangle&=\alpha b(v,w,w)=\alpha\sum_{i=1}^2\int_{\Omega}\frac{\partial v(x)}{\partial x_i}w^2(x)dx \nonumber\\&\leq \alpha\|w\|_{\L^4}^2\|\nabla v\|_{0}\leq \sqrt{2}\alpha\|w\|_{0}\|\nabla w\|_{0}\|\nabla v\|_{0}\nonumber\\
		&\leq\sqrt{\frac{2}{\lambda_1}}\alpha\|\nabla v\|_{0}\|\nabla w\|_{0}^2\leq\sqrt{\frac{2\widetilde{K}}{\lambda_1\nu}}\alpha\|\nabla w\|_{0}^2,
	\end{align}
where $\widetilde{K}$ is defined in \eqref{8.14}.	Thus, combining \eqref{231} and \eqref{237},  we have 
	\[
	\left[\nu-\sqrt{\frac{2\widetilde{K}}{\lambda_1\nu}}\alpha-\frac{\beta}{\lambda_1}(1+\gamma+\gamma^2)\right]\|\nabla w\|_{0}^2+ \beta\|uw\|_{0}^2+\beta\|uw\|_{0}^2\leq 0,
	\]
	and hence the uniqueness follows in this case for $\nu>\sqrt{\frac{2\widetilde{K}}{\lambda_1\nu}}\alpha+\frac{\beta}{\lambda_1}(1+\gamma+\gamma^2)$.
\end{remark}

\begin{remark}\label{rem2.3} For the three-dimensional case, the existence of weak solution to \eqref{8.1} can be 
established for $1\leq \delta<\infty$. Since the proof of  Theorem \ref{thm6.1} involves only interpolation inequalities (see \eqref{215} and \eqref{216}), we infer that  \eqref{8.1} has a weak solution for all $1\leq\delta<\infty$. An application of Sobolev's inequality yields $\H_0^1(\Omega)\subset\L^{2\delta+2}(\Omega)$, for all $1\leq\delta\leq 2$ and hence, in three dimensions, the definition of weak solution given in \eqref{8p3} makes sense for all $v\in\H_0^1(\Omega)\cap\L^{2\delta+2}(\Omega)$, for $2<\delta<\infty$. For the condition given in \eqref{8pp5}, the uniqueness of weak solution follows verbatim as in the proof of Theorem \ref{thm6.2}, since we are only invoking an interpolation inequality (see \eqref{3.26}). 

For $1\leq\delta\leq 2$, the condition given in \eqref{234} needs to be replaced by 
\[
\nu+\frac{\beta\gamma}{\lambda_1}>\frac{\beta}{\lambda_1}2^{2\delta-1}(1+\gamma)^2(\delta+1)^2+\frac{2C\alpha}{\lambda_1^{\frac{2-\delta}{4(\delta+1)}}}\sqrt{\frac{\widetilde{K}}{\beta}},
\]
where $\widetilde{K}$ is defined in \eqref{8.14}. This change is needed since the estimate \eqref{2.31} should be replaced by 
\begin{align*}
-\alpha\langle B(u)-B(v),w\rangle&\leq \alpha\|\nabla w\|_{0}\|w\|_{\L^{2(\delta+1)}}\left(\|u\|_{\L^{2(\delta+1)}}^{\delta}+\|v\|_{\L^{2(\delta+1)}}^{\delta}\right)\nonumber\\
&\leq C\alpha\|\nabla w\|_{0}^{\frac{5\delta+2}{2(\delta+1)}}\|w\|_{0}^{\frac{2-\delta}{2(\delta+1)}}\left(\|u\|_{\L^{2(\delta+1)}}^{\delta}+\|v\|_{\L^{2(\delta+1)}}^{\delta}\right)\nonumber\\&\leq\frac{C\alpha}{\lambda_1^{\frac{2-\delta}{4(\delta+1)}}}\left(\|u\|_{\L^{2(\delta+1)}}^{\delta}+\|v\|_{\L^{2(\delta+1)}}^{\delta}\right)\|\nabla w\|_{0}^2\nonumber\\
&\leq \frac{2C\alpha}{\lambda_1^{\frac{2-\delta}{4(\delta+1)}}} \sqrt{\frac{\widetilde{K}}{\beta}}, \qquad \text{for $1\leq\delta\leq 2$},
\end{align*}
where we have applied Holder's, Gagliardo-Nirenberg's and Young's inequalities. 
\end{remark}
\begin{theorem}[Regularity]
If $\Omega\subset\mathbb{R}^d,d=2,3,$ is either convex, or a domain with $C^2$-boundary and $f\in L^2(\Omega)$, then the weak solution of \eqref{8.1} belongs to $ H^2(\Omega)$. 
\end{theorem}
\begin{proof}
 Let us first assume that $f\in\L^2(\Omega)$. 	Proceeding to multiply \eqref{8p5} by $u_m^{2\delta}\xi_m^k$ and then adding from $k=1,\ldots,m$, we get
 \begin{align*}
 &\nu(2\delta+1)\|u_m^{\delta}\nabla u_m\|_{0}^2+\beta\gamma\|u_m\|_{\L^{2\delta+2}}^{2\delta+2}+\beta\|u_m\|_{\L^{4\delta+2}}^{4\delta+2}\nonumber\\&=\beta(1+\gamma)(u_m^{\delta+1},|u_m|^{2\delta}u_m)+(f,|u_m|^{2\delta}u_m)\nonumber\\
 &\leq\frac{\beta}{2}\|u_m\|_{\L^{4\delta+2}}^{4\delta+2}+\beta(1+\gamma)^2\|u_m\|_{\L^{2\delta+2}}^{2\delta+2}+\frac{1}{\beta}\|f\|_{0}^2,
 \end{align*}
 where we used the Cauchy-Schawrz and Young inequalities. Thus, using \eqref{2p13}, it is immediate to see that 
 \begin{align}\label{242}
 &\nu(2\delta+1)\|u_m^{\delta}\nabla u_m\|_{0}^2+\frac{\beta}{2}\|u_m\|_{\L^{4\delta+2}}^{4\delta+2}\leq(1+\gamma+\gamma^2)\widetilde{K}+\frac{1}{\beta}\|f\|_{0}^2. 
 \end{align} Multiplying \eqref{8p5} by $\lambda_k\xi_m^k$ and then adding from $k=1,\ldots,m$, we can assert that
	\begin{align}\label{8p14}
	\nu\|Au_m \|_{0}^2&=-\alpha (B(u_m ),Au_m )+\beta(C(u_m),Au_m)+(f,Au_m). 
	\end{align}
	Let us take the term $-\alpha (B(u_m ),Au_m )$ from \eqref{8p14} and estimate it using \eqref{242}. Then, H\"older's and Young's inequalities give the following bound
	\begin{align}\label{243}
	\alpha |(B(u_m ),Au_m )|&\leq\alpha \|B(u_m)\|_{0}\|A u_m\|_{0}\leq\alpha\|u_m^{\delta}\nabla u_m\|_{0}\|Au_m\|_{0}\nonumber\\
	&\leq\frac{\nu}{4}\|Au_m\|_{0}^2+\frac{\alpha^2}{\nu}\|u_m^{\delta}\nabla u_m\|_{0}^2.
	\end{align}
	Integrating by parts and applying H\"older's and Young's inequalities,  we find 
	\begin{align*}
	\beta(C&(u_m),Au_m)\nonumber\\
	&= -\beta\gamma\|\nabla u_m\|_{0}^2-\beta(2\delta+1)\|u_m^{\delta}\nabla u_m\|_{0}^2+\beta(1+\gamma)(\delta+1)(u_m^{\delta}\nabla u_m,\nabla u_m)\nonumber\\
	&\leq -\beta\gamma\|\nabla u_m\|_{0}^2-\frac{\beta(2\delta+1)}{2}\|u_m^{\delta}\nabla u_m\|_{0}^2+\frac{\beta(1+\gamma)^2(\delta+1)^2}{2(2\delta+1)}\|\nabla u_m\|_{0}^2. 
	\end{align*}
	Then we use the Cauchy-Schwarz and Young's inequalities to estimate $|(f,Au_m)|$ as 
	\begin{align}\label{2.22}
	|(f,Au_m)|\leq\|f\|_{0}\|Au_m\|_{0}\leq\frac{\nu}{4}\|Au_m\|_{0}^2+\frac{1}{\nu}\|f\|_{0}^2. 
	\end{align}
	Combining \eqref{243}-\eqref{2.22} and substituting the outcome back in \eqref{8p14}, we obtain 
	\begin{align*}
	&	\frac{\nu}{2}\|Au_m \|_{0}^2+\frac{\beta(2\delta+1)}{2}\|u_m^{\delta}\nabla u_m\|_{0}^2\nonumber\\&\leq \frac{\alpha^2}{\nu}\|u_m^{\delta}\nabla u_m\|_{0}^2+\frac{\beta((1+\gamma^2)(\delta+1)^2+2\gamma\delta^2)}{2(2\delta+1)}\|\nabla u_m\|_{0}^2+\frac{1}{\nu}\|f\|_{0}^2.
	\end{align*} 
From the estimates \eqref{2p13} and \eqref{242}, we infer  that $u_m\in D(A)$. Once again invoking the Banach-Alaoglu theorem, we can extract a subsequence $\{u_{m_k}\}$ of $\{u_m\}$ such that 
\begin{equation*}
\left\{\begin{aligned}
 u_{m_k}&\xrightarrow{w}u\ \text{ in }\ L^{4\delta+2}(\Omega)\ \text{ as }\ k\to\infty,\\
 u_{m_k}&\xrightarrow{w}u\ \text{ in }\ D(A)\ \text{ as }\ k\to\infty,
\end{aligned}
\right.\end{equation*} 
since the weak limit is unique. Using the compact embedding of $H^2(\Omega)\subset H^1(\Omega)$, along a subsequence, we further have 
$$u_{m_{k_j}}\to u\ \text{ in }\ H^1(\Omega), \ \text{ as }\ j\to\infty. $$ 
 Proceeding similarly as in the proof of Theorem \ref{thm6.1}, we obtain that $u\in D(A)$  satisfies $$\nu Au +\alpha B(u )-\beta C(u )=f,\ \text{ in }\ L^2(\Omega), $$ and 
 $$\|Au\|_{0}^2+\|u^{\delta}\nabla u\|_0^2+\|u\|_{L^{4\delta+2}}^{4\delta+2}\leq C(\|f\|_0,\nu,\alpha,\beta,\gamma,\delta).$$ 
  But, we know that 
 \begin{align*}
& \|f-\alpha B(u )+\beta C(u )\|_0\\&\leq \|f\|_0+\alpha\|u^{\delta}\nabla u\|_0+\beta\gamma\|u\|_0+\beta(1+\gamma)\|u\|_{L^{2\delta+2}}^{\delta+1}+\beta\|u\|_{L^{4\delta+2}}^{2\delta+1}<\infty, 
  \end{align*}
  and hence an application of \cite[Th. 9.25]{brezis_book11} (for a domain with $C^2$ boundary) or \cite[Th. 3.2.1.2]{grisvard_book85} (for convex domains) yields $u\in H^2(\Omega)$. 
\end{proof}


\section{Numerical schemes and their a priori error estimates}\label{sec3}\setcounter{equation}{0} 
Let the domain $\Omega$ be partitioned into a mesh (consisting of shape-regular triangular or rectangular 
cells $K$) denoted by $\mathcal{T}_h$. We use the symbols $\mathcal{E}_h$, $\mathcal{E}^i_h$ and $\mathcal{E}^{\partial}_h$ to denote  the set of edges,  interior edges and 
 boundary edges of the mesh, respectively. For a given $\mathcal{T}_h$, 
the notations $C^{0}(\mathcal{T}_h)$ and $H^s(\mathcal{T}_h)$ indicate  broken spaces associated with 
continuous and differentiable function spaces, respectively.
\subsection{Conforming method} 
	Let $V_h$ be a finite dimensional subspace of $\H_0^1(\Omega)$ associated with the mesh parameter $h$. Numerical solutions are sought in the family $\{V_h\}\subset\H_0^1(\Omega),$ (where one additionally assumes that $h$ is sufficiently small) satisfying the following approximation 
	property (see \cite{VTh})
\begin{align*}
\inf_{\chi\in V_h}\left\{\|u-\chi\|_{0}^2+h\|\nabla(u-\chi)\|_{0}^2\right\}\leq Ch^k\|u\|_{k},
\end{align*}
for all $u\in\H^r(\Omega)\cap\H_0^1(\Omega)$, $ 1\leq k\leq r$, where $r$ is the order of accuracy of the family $\{V_h \}$. 
The CFEM for \eqref{8p2} reads: find $u_h\in V_h$ such that 
\begin{equation}\label{7p1}
\nu a(u_h,\chi)+\alpha b(u_h,u_h,\chi)=\beta\langle C(u_h),\chi\rangle+\langle f,\chi\rangle, \qquad \forall \chi \in V_h.
\end{equation}

\begin{theorem}[Existence of a discrete solution]\label{excgfem}
 Equation (\ref{7p1}) admits at least one solution $u_h\in V_h$.
\end{theorem}
\begin{proof}
It follows as a direct consequence of Theorem \ref{thm6.1}.
\end{proof}
Let  $R^h$ be the elliptic or Ritz projection  onto $V_h$ (see \cite{VTh}), defined by  
\begin{align*}
(\nabla R^hv,\nabla\chi)=(\nabla v,\nabla\chi), \text{ for all }\ \chi\in V_h \ \text{ for }\ v\in\H_0^1(\Omega). 
\end{align*}
By setting $\chi=R^hv$ above,  we readily obtain  that the Ritz projection is stable, that is, $\|\nabla R^hv\|_{0}\leq\|\nabla v\|_{0}$, for all $v\in\H_0^1(\Omega)$. Moreover, using \cite[Lem. 1.1]{VTh}, we have 
\begin{align}\label{7a1}
\|R^hv-v\|_{0}+h\|\nabla(R^hv-v)\|_{0}\leq Ch^s\|v\|_{s},
\end{align}
for all $ v\in\H^s(\Omega)\cap\H_0^1(\Omega)$, $1\leq s\leq r$.

	\begin{theorem}[Energy estimate]\label{thm7.1}
	Let $V_h$ be a finite dimensional subspace of $\H_0^1(\Omega)$. Assume that \eqref{234} holds true and that $u\in D(A)=\H_0^1(\Omega)\cap \H^2(\Omega)$ satisfies \eqref{8p2}. Then the error incurred by the Galerkin approximation satisfies 
	\[
	\|u_h-u\|_{1}\leq Ch,
	\]
	where $C$ is a constant possibly depending on $\nu,\alpha,\beta,\gamma,\delta$, $\|f\|_{0}$, but independent of $h$.   	
\end{theorem}
\begin{proof}
Using triangle inequality we can write 
\begin{equation}\label{7p7}
\|u_h-u\|_{1}\le \|u_h-W\|_{1}+\|W-u\|_{1},
\end{equation}
where $W\in V_h$. We need to estimate $\|u_h-W\|_{1}$. First we note that from \eqref{7a1}, the second term in the RHS of \eqref{7p7} satisfies 
\[
\|W-u\|_{1}\le Ch.
\]
Next, and using \eqref{8p3} and \eqref{7p1}, we can assert that $u^h-u$ satisfies
	\begin{align}\label{7p2}
	\nu a(u_h-u,\chi)= -\alpha[b(u_h,u_h,\chi)-b(u,u,\chi)]+\beta[\langle C(u_h),\chi\rangle -\langle C(u),\chi\rangle], 
	\end{align}
	for all $\chi\in V_h$. Let us choose $\chi=u_h-W\in V_h$ in \eqref{7p2}, to eventually obtain 
	\begin{align}\label{7p3}
	\nu a(u_h-u,u_h-W)&= -\alpha[b(u_h,u_h,u_h-W)-b(u,u,u_h-W)]\nonumber\\
	&\quad+\beta[\langle C(u_h),u_h-W\rangle -\langle C(u),u_h-W\rangle ]. 
	\end{align}
On the other hand, we can write $u_h-u$ as $u_h-W+W-u$ in \eqref{7p3} to find
	\begin{align*}
	\nu\|\nabla(u_h-W)\|_{0}^2&=-\nu(\nabla(W-u),\nabla \chi)-\alpha[b(u_h,u_h,\chi)-b(W,W,\chi)]\\
	&\quad-\alpha[b(W,W,\chi)-b(u,u,\chi)]+\beta[\langle C(u_h),\chi\rangle-\langle C(W),\chi\rangle ]\\
	&\quad+\beta[\langle C(W),\chi\rangle -\langle C(u),\chi\rangle ].
	\end{align*}	
	Thus, following \eqref{327} and \eqref{6.49}, we can establish the bound 
	\begin{align}\label{7p5}
	\frac{\nu}{2}\|\nabla \chi\|_{0}^2+\left(\frac{\beta}{4}-\frac{4^{\delta}\alpha^2}{4\nu}\right)&\|{u_h}^{\delta}\chi\|_{0}^2+\left(\frac{\beta}{4}-\frac{4^{\delta}\alpha^2}{4\nu}\right)\|W^{\delta}\chi\|_{0}^2\nonumber\\
	+(\beta\gamma-C(\beta,\alpha,\delta))\|\chi\|_{0}^2&
	\leq \nu(\nabla(u-W),\nabla \chi)-\alpha\sum_{i=1}^2\left({W}^{\delta}\frac{\partial W}{\partial x_i}-u^{\delta}\frac{\partial u}{\partial x_i},\chi\right)\nonumber\\
	&\quad+\beta(W(1-{W}^{\delta})({W}^{\delta}-\gamma)-u(1-u^{\delta})(u^{\delta}-\gamma),\chi),
	\end{align}
	where we have introduced the constant $C(\beta,\alpha,\delta)= \beta 2^{2\delta-1}(1+\gamma)^2(\delta+1)^2$. 
	Using an integration by parts, Taylor's formula, H\"older's and Young's inequalities, we can rewrite the first term on 
	the RHS of \eqref{7p5} as 
	\begin{align}\label{7p6}
	-\frac{\alpha}{\delta+1}&\sum_{i=1}^d\left(\frac{\partial}{\partial x_i}({W}^{\delta+1}-u^{\delta+1}),\chi\right)
	=\frac{\alpha}{\delta+1}\sum_{i=1}^d({W}^{\delta+1}-u^{\delta+1},\frac{\partial}{\partial x_i}\chi)\nonumber\\
	&=\alpha\sum_{i=1}^d\left((\theta W+(1-\theta)u)^{\delta}(W-u),\frac{\partial}{\partial x_i}\chi\right)\nonumber\\
	&\leq 2^{\delta-1}\alpha\left(\|{W}^{\delta}(W-u)\|_{0}+\|{u}^{\delta}(W-u)\|_{0}\right)\|\nabla \chi\|_{0}\nonumber\\
	&\leq 2^{\delta-1}\alpha\left(\|{W}^{2\delta}\|_{0}^{1/2}+\|{u}^{2\delta}\|_{0}^{1/2}\right)\|W-u\|_{L^4}\|\nabla \chi\|_{0}.
	\end{align}
	And we can also rewrite the second term on the RHS of \eqref{7p5} as 
	\[
	\beta(1+\gamma)({W}^{\delta+1}-u^{\delta+1},\chi)-2\beta\gamma(W-u,\chi)
	-2\beta({W}^{2\delta+1}-u^{2\delta+1},\chi):=\sum_{i=1}^3J_i,
	\]
	where
	\begin{gather*}
	J_1=\beta(1+\gamma)({W}^{\delta+1}-u^{\delta+1},\chi),\qquad 
	J_2 =-2\beta\gamma(W-u,\chi),\\
	J_3=-2\beta({W}^{2\delta+1}-u^{2\delta+1},\chi). 
	\end{gather*}
	We estimate $J_1$ using Taylor's formula, H\"older's and Young's inequalities as 
	\begin{align*}
	J_1&=
	\beta(1+\gamma)(\delta+1)((\theta W+(1-\theta)u)^{\delta}(W-u),\chi)\\
	&\leq 2^{\delta-1}\beta(1+\gamma)(\delta+1)\left(\|{W}^{\delta}(W-u)\|_{0}+\|{u}^{\delta}(W-u)\|_{0}\right)\|\chi\|_{0}\\
	&\leq 2^{\delta-1}\beta(1+\gamma)(\delta+1)\left(\|{W}^{2\delta}\|_{0}^{1/2}+\|{u}^{2\delta}\|_{0}^{1/2}\right)\|W-u\|_{L^4}\|\chi\|_{0}.
	\end{align*}
	In turn, using Cauchy-Schwarz and Young's inequalities, an estimate for $J_2$ reads 
	\[
	J_2 \leq 2\beta\gamma\|W-u\|_{0}\|\chi\|_{0},
	\]
	while a bound for $J_3$ results from applying 
	Taylor's formula together with H\"older's and Young's inequalities 
	\begin{align}\label{7p11}
	J_3&=-(2\delta+1)\beta((\theta W+(1-\theta)u)^{2\delta}(W-u),\chi)\nonumber\\
	&\leq 2^{2\delta-1}(2\delta+1)\beta \left(\|{W}^{\delta}(W-u)\|_{0}\|{W}^{\delta}\chi\|_{0}+\|{u}^{\delta}(W-u)\|_{0}\|{u}^{\delta}\chi\|_{0}\right)\nonumber\\
	&\leq 2^{2\delta-1}(2\delta+1)\beta \left(\|{W}^{2\delta}\|_{0}+\|{u}^{2\delta}\|_{0}\right)\|W-u\|_{L^4}\|\chi\|_{L^4}.
	\end{align}
	Combining \eqref{7p6}-\eqref{7p11}, substituting the result back into \eqref{7p5}, and then using 
	\eqref{7a1} and \eqref{7p7}, implies the desired result. 
\end{proof}
\subsection{Non-conforming finite element method}
Let $\mathbb{P}_1$  denote the space of polynomials which have degree at most $1$, and let us recall the 
definition of the Crouzeix-Raviart (CR) non-conforming finite element space 
\begin{equation}\label{CRFEM11}
{V}_{h}^{CR} =\left\{v\in L^2(\Omega) :\ \text{ for all }\  K\in \mathcal{T} \;v_{|_K}\in\mathbb{P}_1 \;\mbox{and}\; \int_E [|v|]=0\quad E\in\mathcal{E}\right\}.
\end{equation}
It is useful to introduce the piecewise gradient operator $\nabla_h: H^1(\mathcal{T}_h)\rightarrow L^2(\Omega;\mathbb{R}^2)$ with $(\nabla_h v)|_K = \nabla v|_K,$ for all $K\in \mathcal{T}_h$. 
The discrete weak formulation of (\ref{8.1}) in this context reads: find $u^{CR}_h\in V_h^{CR}$ such that
\begin{align}\label{ncweakform}
A_{NC}(u_h^{CR},\chi)=(f,\chi), \quad \ \text{ for all }\  \chi\in V_h^{CR},
\end{align} 
with 
\begin{gather*}
A_{NC}(v,v)=\nu a_{NC}(v,v)+\alpha b_{NC}(v;v,v)-\beta(C(v),v),\\
a_{NC}(v,v) = (\nabla_h v, \nabla_h v), \quad b_{NC}(v;v,v)= ((v^{\delta},v^{\delta})^T\cdot\nabla_h v,v),
\end{gather*}
and we define the associated discrete energy norm $\vertiii{v}_{NC}:=\sqrt{a_{NC}(v,v)}$.
\begin{lemma}\label{crnclem11}
For any $v\in V_h^{CR}$, we have
\begin{align}\label{317}
A_{NC}(v,v)\ge \bar{C} \vertiii{v}_{NC}^2,
\end{align}
provided $\nu > \max\{{\beta}(1+\gamma^2)C_{\Omega}^{NC},\frac{2\alpha^2}{\beta}\}$.
\end{lemma}
\begin{proof}
Owing to Young's and Poincar\'{e}-Friedrichs's inequalities, it readily follows that
\begin{align*}
	A_{NC}(v,v)&=\nu\|\nabla_h v\|^2_{0,\mathcal{T}_h}+\beta\gamma\|v\|_{0}^2+\beta\|v\|_{\L^{2\delta+2}}^{2\delta+2}-\beta(1+\gamma)(v^{\delta+1},v)-b_{NC}(v;v,v)\nonumber\\
	&\geq \nu\|\nabla_h v\|_{0,\mathcal{T}_h}^2+\beta\gamma\|v\|_{0}^2+\beta\|v\|_{\L^{2\delta+2}}^{2\delta+2}-\beta(1+\gamma)\|v\|_{\L^{\delta+1}}^{\delta+1}\|v\|_{0}\nonumber\\
	&\qquad-\alpha\|v\|_{L^{\delta+1}}^{\delta+1}\|\nabla_h v\|_{0,\mathcal{T}_h}\nonumber\\
	&\geq {\nu}\|\nabla_h v\|_{0,\mathcal{T}_h}^2+\beta\gamma\|v\|_{0}^2+\frac{\beta}{4}\|v\|_{\L^{2\delta+2}}^{2\delta+2}-\frac{\beta}{2}(1+\gamma)^2\|v\|_{0}^2-\frac{\alpha^2}{\beta}\|\nabla_hv\|_{0,\mathcal{T}_h}^2\nonumber\\
	&\geq\nu\|\nabla_h v\|_{0,\mathcal{T}_h}^2-\frac{\beta}{2}(1+\gamma^2)\|v\|_{0}^2-\frac{\alpha^2}{\beta}\|\nabla_hv\|_{0,\mathcal{T}_h}^2\nonumber\\
	&\geq\left(\frac{\nu}{2}-\frac{\beta}{2}(1+\gamma^2)C_{\Omega}^{NC}+\frac{\nu}{2}-\frac{\alpha^2}{\beta}\right)\|\nabla_h v\|_{0,\mathcal{T}_h}^2,
	\end{align*}
	and the estimate \eqref{317} follows. 
\end{proof}
\begin{theorem}[Existence of a discrete solution]\label{excrfem}
 Let $\|u_h^{CR}\|_0=k_{CR}$ and 
 \begin{align*}
 k_{CR}> \frac{(C_{\Omega}^{CR})}{\nu\sqrt{\nu+\beta\gamma C_{\Omega}^{CR}-\beta(1+\gamma)^2C_{\Omega}^{CR}-\frac{2\alpha^2}{\beta}}}\|f\|_{0}, 
 \end{align*} 
 provided $\nu +\beta \gamma C_{\Omega}^{CR}> \beta(1+\gamma)^2C_{\Omega}^{CR}+\frac{2\alpha^2}{\beta}$. Then, problem (\ref{ncweakform}) admits at least one solution $u_h^{NC}\in V_h^{NC}$.
\end{theorem}
\begin{proof}
 We introduce the Crouzeix-Raviart operator $P_{CR}:V_h^{CR}\rightarrow V_h^{CR}$ as 
 \[
 (P_{CR}(u_{h}^{CR}),v)=A_{NC}(u_{h}^{CR},v)-(f,v),
 \]
which is well defined and continuous on $V_h^{CR}$. Choosing $v=u_h^{CR}$ and using Lemma \ref{crnclem11}, we have
 \begin{align}\label{eqcrnc12}
 (P_{CR}(u_h^{CR}),&u_h^{CR})\nonumber\\
 &\ge \nu\|\nabla_h v\|_{0,\mathcal{T}_h}^2-\frac{\beta}{2}(1+\gamma^2)\|v\|_{0}^2-\frac{\alpha^2}{\beta}\|\nabla_hv\|_{0,\mathcal{T}_h}^2+\beta\gamma\|v\|_{0}^2\nonumber\\
 &\quad-\frac{C_{\Omega}^{CR}}{2\nu}\|f\|_0^2-\frac{\nu}{2C_{\Omega}^{CR}}\|u^{CR}_h\|_0^2,\nonumber\\
 &\ge\frac{1}{C_{\Omega}^{CR}} \left(\frac{\nu}{2}-\frac{\beta}{2}(1+\gamma^2)C_{\Omega}^{CR}-\frac{\alpha^2}{\beta}+\beta\gamma C_{\Omega}^{CR}\right)\| u_{h}^{CR}\|_{0}^2-\frac{C_{\Omega}^{CR}}{2\nu}\|f\|_0^2.
 \end{align}
 Let $\|u_h^{CR}\|_0=k_{CR}$ and 
 \[
 k_{CR}> \frac{(C_{\Omega}^{CR})}{\nu\sqrt{\nu+\beta\gamma C_{\Omega}^{CR}-\beta(1+\gamma)^2C_{\Omega}^{CR}-\frac{2\alpha^2}{\beta}}}\|f\|_{0}, 
\]
 provided $\nu +\beta \gamma C_{\Omega}^{CR}> \beta(1+\gamma)^2C_{\Omega}^{CR}+\frac{2\alpha^2}{\beta}$. Then the RHS in (\ref{eqcrnc12}) is non-negative. Finally,  
Brouwer's fixed-point theorem implies that $P_{CR}(u_h^{CR})=0$.
\end{proof}

Next we denote by $I_h$ the usual finite element interpolation \cite{john1998}. Then the following estimates hold
\begin{align}\label{ncapprox11}
|v-I_hv|_{m,k}&\le Ch^{2-m}_K\|v\|_{2,K}\quad v\in H^2(K),\\
\|v-(I_hv)\|_{0,E}&\le Ch^{3/2}\|v\|_{2,K}\quad v\in H^2(K)\quad E\in\mathcal{E}(\mathcal{T}_h). \label{ncapprox12}
\end{align}
Regarding the edge projection $P_E:L^2(E)\rightarrow P_0(E)$, where $P_0(E)$ is a constant on $E$, we have 
\begin{align}\label{P0projection}
\|v-P_E v\|_{0,E}\le Ch^{1/2}_K|v|_{1,K},\ \text{ for all }\  v\in H^1(K),\ E\in\mathcal{E}(\mathcal{T}_h).
\end{align}
\begin{lemma}\label{crnclem111}
There holds:
\begin{align*}
\alpha[b_{NC}(v_1,v_1,w)-b_{NC}(v_2,v_2,w)]&\le \frac{\nu}{2}\|\nabla_h w\|_{0,\mathcal{T}_h}^2+\frac{2^{2\delta}C_{\star}\alpha^2}{4\nu}(\|v^{\delta}_1 w\|_{0}^2+\|v^{\delta}_2 w\|_{0}^2),\\
A_{NC}(v_1,w)-A_{NC}(v_2,w)&\ge \frac{\nu}{2}\|\nabla_h w\|_{0,\mathcal{T}_h}^2+(\beta\gamma-C(\beta,\alpha,\delta))\|w\|_{0}^2\nonumber\\
&+\left(\frac{\beta}{4}-\frac{2^{2\delta}C_{\star}\alpha^2}{4\nu}\right)(\|{v}^{\delta}_1w\|_{0}^2+\|v^{\delta}_2w\|_{0}^2),
\end{align*}
where $v_1,v_2\in V_{h}^{NC}$, $w=v_1-v_2$ and $C_{\star}$ is a postive constant.
\end{lemma}
\begin{proof}
To prove the first estimate, we use the definition of $b_{NC}(\cdot, \cdot)$. Then
\begin{align*}
\alpha[b_{NC}(v_1,v_1,w)-b_{NC}(v_2,v_2,w)]&=\alpha \sum_{K\in\mathcal{T}_h}\sum_{i=1}^{d}\int_{K}\left(v_1^{\delta}\frac{\partial v_1}{\partial x_i} -v_2^{\delta}\frac{\partial v_2}{\partial x_i}\right)w dx\\
&=\frac{\alpha}{\delta+1}\sum_{K\in\mathcal{T}_h}\sum_{i=1}^{d}\int_{K}\left(\frac{\partial (v_1^{\delta+1}- v_2^{\delta+1})}{\partial x_i} \right)w dx. 
\end{align*}
Using Cauchy-Schwarz and inverse inequalities, Taylor's formula,
H\"{o}der's and Young's inequalities, implies the first stated result. 
To prove the second inequality, we write 
\begin{align*}
A_{NC}(v_1,w)-A_{NC}(v_2,w)&= \nu a_{NC}(v_1-v_2,w)+\alpha[b_{NC}(v_1,v_1,w)-b_{NC}(v_2,v_2,w)]\\
&\quad-\beta[(C(v_1),w)-(C(v_2),w)].
\end{align*}
Applying the first estimate and \eqref{327} leads to the second estimate.
\end{proof}

\begin{theorem}\label{ncthm11}
	Let $V_h^{CR}$ be the non-conforming space defined in (\ref{CRFEM11}). Assume that \eqref{234} holds true and  that $u\in D(A)=\H_0^1(\Omega)\cap \H^2(\Omega)$ satisfies \eqref{8p2}. Then the error incurred by the NCFEM approximation satisfies 
	\[
	\vertiii{u^{CR}_h-u}_{NC}\leq Ch,
	\]
	where the constant $C$ is independent of $h$ and $C$ depends on  	$\nu,\alpha,\beta,\gamma,\delta$, $\|f\|_{0}$, etc.
\end{theorem}
\begin{proof} Similarly as before, we split the error and use triangle inequality to write
\[
\vertiii{u_h^{CR}-u}_{NC}\le \vertiii{ u_h^{NC}-W}_{NC}+\vertiii{ W-u}_{NC}.
\]
From \eqref{ncapprox11},  the following estimate is valid for the second term on the RHS 
\[
\vertiii{ W-u}_{NC}\le Ch.
\]
Using (\ref{ncweakform}), we have
\[
A_{NC}(u_h^{CR},\chi)=(f,\chi), \ \text{ for all }\ \chi\in V_h^{CR}.
\]
If $u\in D(A)=\H_0^1(\Omega)\cap \H^2(\Omega)$ satisfies \eqref{8p2}, then it readily 
follows that
\[
A_{NC}(u,\chi)=(f,\chi)+\sum_{K\in\mathcal{T}}\int_{K}\nu\frac{\partial u}{\partial n_K} \chi, \ \text{ for all }\ \chi\in V_h^{CR}.
\]
We can then use Lemma \eqref{crnclem111}, which leads to 
\begin{align*}
 &\frac{\nu}{2}\|\nabla_h \chi\|_{0,\mathcal{T}_h}^2+(\beta\gamma-C(\beta,\alpha,\delta))\|\chi\|_{0}^2+\left(\frac{\beta}{4}-\frac{2^{2\delta}C_{\star}\alpha^2}{4\nu}\right)(\|{u}^{CR}_h\chi\|_{0}^2+\|W^{\delta}\chi\|_{0}^2)\\
 &\leq A_{NC}(u,\chi)-A_{NC}(W,\chi)-\sum_{K\in\mathcal{T}}\int_{K}\nu\frac{\partial u}{\partial n_K} \chi.
\end{align*}
To estimate the consistency error, it suffices to exploit the CR approximation 
\[
\sum_{K\in\mathcal{T}}\int_{\partial K}\nu\frac{\partial u}{\partial n_K} \chi=-\sum_{E\in\mathcal{E}}\int_{E}\nu\frac{\partial u}{\partial n_E} [\chi] 
=-\sum_{E\in\mathcal{E}}\int_{E}\nu\left(\frac{\partial u}{\partial n_E} -P\left(\frac{\partial u}{\partial n_E}\right)\right)[\chi].
\]
Consequently, we can invoke estimate (\ref{P0projection}), which yields 
\[
\left|\sum_{K\in\mathcal{T}}\int_{\partial K}\nu\frac{\partial u}{\partial n_K} \chi\right| \le C \left(\sum_{K\in\mathcal{T}}\nu h_K^2\|u\|_{2,K}^2\right)^{1/2}\vertiii{ \chi}_{NC},
\]
and the remainder of the proof follow similarly to that of Theorem \ref{thm7.1}.
\end{proof}

\subsection{Discontinuous Galerkin method}
In addition to the mesh notation used so far, we also require the 
following preliminaries. 
 Let $E=K_+\cap K_-\in\mathcal{E}^i_h$ be the common edge that is shared by the 
two mesh cells $K_\pm$. We use the symbol $w_{\pm}$ to denote 
 the traces of functions $w\in C^0(\mathcal{T}_h)$ on $E$ from $K_\pm$,
 respectively. In addition, we denote 
the sum (which in turn translates into the jump operator) over an edge as 
\begin{align*}
[\![w]\!]=w_+ + w_-,
\end{align*}
and if ${w}\in C^1(\mathcal{T}_h)$ we also define 
\begin{align*}
[\![\partial{w}/\partial\bm{n}]\!]=\nabla({w}_+ -{w}_-)\bm{n}_+, 
\quad\text{and}\quad[\![{w}\otimes{n}]\!]=({w}_+ - {w}_-)\otimes\bm{n}_+,
\end{align*}
where $\bm{n}_\pm$ denote the unit outward normal vectors to $K_\pm$, respectively. In case of 
boundary edges $E=K_+\cap\partial\Omega$, we take $[\![{v}]\!]={w}_+$.  The exterior trace of $u$ taken over the edge under consideration is denoted by $u^e$ and we chose $u^e =0$ for boundary edges.
We recall the definition of  
the local gradient 
$\nabla_h$ 
satisfying $(\nabla_h{w})|_K = \nabla({w}|_K)$ on each $K\in\mathcal{T}_h$.
We will use the discrete subspace of $L^2(\Omega)$ 
\begin{align}\label{dgsubspace1}
{V}_h^{DG}=\{{v}\in L^2(\Omega): \text{ for all }\  K\in \mathcal{T}_h : {v}|_K \in \mathcal{P}_1(K)\}.
\end{align}
where $\mathcal{P}_1(K)$ is the space of polynomials on $K$ having partial degree $1$.

The discrete weak formulation of (\ref{8.1}) reads now: find $u^{DG}_h\in V_h^{DG}$ such that
\begin{equation}\label{dgweakform}
A_{DG}(u_h^{DG},\chi)=(f,\chi), \ \text{ for all }\ \chi\in V_h^{DG},
\end{equation} 
where, for ${u},{v}\in {V}^{DG}_h$,  the bilinear form 
\begin{equation}\label{ADG}
A_{DG}(v,v)=\nu a_{DG}(v,v)+\alpha b_{DG}(\bm{v},v,v)-\beta(C(v),v),\end{equation}
is defined with the following contributions
\begin{gather*}
a_{DG}({u},{v})=(\nabla_h {u}, \nabla_h {v})+a^{i}_{h}({u},{v})+a^{\partial}_{h}({u},{v}),\\
a^{i}_{h}({u},{v})=a^{i}_p({u},{v})-a^{i}_c({u},{v})-a^{i}_c({v},{u}),\quad 
a^{\partial}_{h}({u},{v})=a^{\partial}_p({u},{v})-a^{\partial}_c({u},{v})-a^{\partial}_c({v},{u}),\\
a^{i}_c({u},{v})=\frac{1}{2}\sum_{E\in\mathcal{E}^i_h}\int_E [\![\nabla_h {u}]\!]\!\cdot\![\![{v}\otimes \bm{n}]\!]d{s},\quad 
 a^{i}_p({u},{v})=\sum_{E\in\mathcal{E}^i_h}\int_E\gamma_h [\![{u}\otimes \bm{n}]\!]\!\cdot\![\![{v}\otimes \bm{n}]\!]d{s},\\
a^{\partial}_c({u},{v})= \sum_{E\in\mathcal{E}^{\partial}_h}\int_E \nabla {u}\!\cdot\!({v}\otimes \bm{n})d{s},
\quad a^{\partial}_p({u},{v})=2\sum_{E\in\mathcal{E}^{\partial}_h}\int_E\gamma_h ({u}\otimes \bm{n})\!\cdot\!({v}\otimes \bm{n})d{s},\\
b_{DG}(\bm{w};u,v)= \!\sum_{K\in \mathcal{T}_h} \!\!\int_{K} \!\!\bm{w}\cdot \nabla u {v} dx 
+\!\!\!\sum_{K\in \mathcal{T}_h} \frac{1}{2}\int_{\partial K} \left[\bm{w}\cdot\bm{n}_K (u^e \! -\! u)\!-|\bm{w}\cdot \bm{n}_K|(u^e\!-\!u)\right]{v} ds,
\end{gather*}
with $\bm{w}=(w,w)^T$ and $\gamma_h=\frac{\gamma}{h_E}$, where $h_E$ is the length of the edge $E$ and $\gamma$ is a penalty parameter chosen sufficiently large to guarantee the stability of the formulation (see, e.g.,  \cite{ADN}). 

It is also convenient to rewrite $b_{DG}(\cdot;\cdot,\cdot)$, after integration by parts, as follows
\begin{align*}
b_{DG}(\bm{w};u,v)&= \sum_{K\in \mathcal{T}_h} \int_{K} (-u \bm{w}\cdot \nabla {v} -\nabla\cdot \bm{w} u {v})dx\\
 &\quad+\sum_{K\in \mathcal{T}_h} \int_{\partial K} \left[\frac{1}{2}\bm{w}\cdot\bm{n}_K [\![u]\!]-\frac{1}{2}|\bm{w}\cdot \bm{n}_K|(u^e-u)\right]{v} ds.
\end{align*}

For the subsequent error analysis, we adopt the following discrete norm 
\[
\vertiii{ v}^2:= \sum_{K\in\mathcal{T}_h}\|\nabla_h v\|_{0,K}^2+\sum_{E\in\mathcal{E}(\mathcal{T}_h)}\|[\![v]\!]\|_{0,E}^2.\]

 \begin{lemma}\label{dglem11}
Coercivity of $a_{DG}$ and continuity of $b_{DG}$ hold in the following sense 
 \[
 a_{DG}(v,v) \ge \alpha_a\vertiii{ v}^2, \quad 
 \alpha b_{DG}(\bm{v};v,v)\le  \frac{\beta}{4}\|v\|_{L^{2\delta+2}}^{2\delta+2}+\frac{2\alpha^2}{\beta}\vertiii{v}^2, \qquad \forall 
 v\in V_{h}^{DG}.\]
 \end{lemma}
 \begin{proof}
 The first estimate follows from \cite{ADN}.  
 Using Cauchy-Schwarz, inverse trace and Young's inequalities in $b_{DG}$, implies the second stated result. 
 \end{proof}
 
\begin{lemma}
For any $v\in V_h^{DG}$, the bilinear form $A_{DG}$ defined in \eqref{ADG} satisfies
\[
A_{DG}(v,v)\ge \bar{C} \vertiii{v}^2.
\]
\end{lemma}
\begin{proof}Owing to Young's inequality and Lemma \ref{dglem11}, we have
\begin{align*}
	A_{DG}(v,v)&\ge \alpha_a\nu \vertiii{ v} ^2+\beta\gamma\|v\|_{0}^2+\beta\|v\|_{\L^{2\delta+2}}^{2\delta+2}-\beta(1+\gamma)(v^{\delta+1},v)-\alpha b_{DG}(\bm{v};v,v)\nonumber\\
	&\geq \alpha_a{\nu}\vertiii{ v} ^2+\beta\gamma\|v\|_{0}^2+\frac{\beta}{4}\|v\|_{\L^{2\delta+2}}^{2\delta+2}-\frac{\beta}{2}(1+\gamma)^2\|v\|_{0}^2-\frac{2\alpha^2}{\beta}\vertiii{ v}^2\nonumber\\
	&\geq \alpha_a\nu\vertiii{ v}^2-\frac{\beta}{2}(1+\gamma^2)\|v\|_{0}^2-\frac{2\alpha^2}{\beta}\vertiii{ v}^2\nonumber\\
	&\geq\left(\frac{\alpha_a\nu}{2}-\frac{\beta}{2}(1+\gamma^2)C_{\Omega}+\frac{\alpha_a\nu}{2}-\frac{2\alpha^2}{\beta}\right)\vertiii{ v}^2.
	\end{align*}
\end{proof}
\begin{theorem}[Existence of a discrete solution]
 Let $\|u^{DG}_h\|_0=k_{DG}$ and 
 \begin{align*}
 k_{DG}> \frac{(C_{\Omega}^{DG})}{\nu\sqrt{\nu+\beta\gamma C_{\Omega}^{DG}-\beta(1+\gamma)^2C_{\Omega}^{DG}-\frac{2\alpha^2}{\beta}}}\|f\|_{0}, 
 \end{align*}
 provided $\nu +\beta \gamma C_{\Omega}^{DG}> \beta(1+\gamma)^2C_{\Omega}^{DG}+\frac{2\alpha^2}{\beta}$. Then equation (\ref{dgweakform}) admits at least one solution $u_h^{DG}\in V_h^{DG}$.
\end{theorem}
\begin{proof}
 Proceeding as before, we introduce the map 
 $P_{DG}:V_h^{DG}\rightarrow V_h^{DG}$ with
 \[
 (P_{DG}(u_{h}^{DG}),v)=A_{DG}(u_{h}^{DG},v)-(f,v),
 \]
 which is well-defined and continuous. 
 Choosing $v=u_h^{DG}$ in Lemma \ref{dglem11} yields 
 \begin{align}\label{eqdg12}
 (P_{DG}(u_h^{DG}),&u_h^{DG})\nonumber\\
 &\ge \alpha_a\nu \vertiii{ u_h^{DG}}^2-\frac{\beta}{2}(1+\gamma^2)\|u^{DG}_h\|_{0}^2-\frac{2\alpha^2}{\beta}\vertiii{ u_h^{DG}} ^2+\beta\gamma\|u^{DG}_h\|_{0}^2\nonumber\\
 &\quad-\frac{C_{\Omega}^{DG}}{2\nu}\|f\|_0^2-\frac{\nu}{2C_{\Omega}^{DG}}\|u^{DG}_h\|_0^2,\nonumber\\
 &\ge\frac{\alpha_a}{C_{\Omega}^{DG}} \left(\frac{\nu}{2}-\frac{\beta(1+\gamma^2)C_{\Omega}^{DG}}{2\alpha_a}-\frac{\alpha^2}{\beta\alpha_a}+\frac{\beta\gamma C_{\Omega}^{DG}}{\alpha_a}\right)\| u_{h}^{DG}\|_{0}^2-\frac{C_{\Omega}^{DG}}{2\nu}\|f\|_0^2.
 \end{align}
 Next, let us define $\|u_h^{DG}\|_0=k_{DG}$, and note that  
 \begin{align*}
 k_{DG}> \frac{(C_{\Omega}^{DG})}{\nu\sqrt{\alpha_a\nu+2\beta\gamma C_{\Omega}^{DG}-\beta(1+\gamma)^2C_{\Omega}^{DG}-\frac{2\alpha^2}{\beta}}}\|f\|_{0}, 
 \end{align*} 
 provided that $\nu +2\beta \gamma C_{\Omega}^{DG}> \beta(1+\gamma)^2C_{\Omega}^{DG}+\frac{2\alpha^2}{\beta}$. Then the RHS in (\ref{eqdg12}) is non-negative. Finally,  
Brouwer's fixed point theorem implies that $P_{DG}(u_h^{DG})=0$.
\end{proof}

On the other hand, we can 
establish the following result, whose proof is similar to  \eqref{crnclem111}.

\begin{lemma}\label{dglem111}
There holds:
\[
A_{DG}(v_1,w)-A_{DG}(v_2,w)\ge \tilde{C}_{DG}\vertiii{ w} ,
\]
where $v_1,v_2\in V_{h}^{DG}$ and $w=v_1-v_2$.
\end{lemma}
Finally, we can state an a priori error estimate in the following theorem.
\begin{theorem}\label{dgerr11}
	Let $V_h^{DG}$ be as in (\ref{dgsubspace1}), and let us assume \eqref{234} and that $u$ satisfies \eqref{8p2}. Then,  there exists $\tilde{C}$ is independent of $h$ such that 
	\[
	\vertiii{ u^{DG}_h-u|} \leq \tilde{C}h.
	\]
\end{theorem}
\begin{proof}
Using triangle inequality readily gives
\[
\vertiii{ u_h^{DG}-u} \le \vertiii{ u_h^{DG}-W} +\vertiii{W-u}.
\]
Proceeding again as in the conforming and non-conforming cases, we have the bound 
\[
\vertiii{ W-u} \le Ch.
\]
Using the formulation \eqref{dgweakform}, we have
\[
A_{DG}(u_h^{DG},\chi)=(f,\chi), \quad \ \text{ for all }\  \chi\in V_h^{DG},
\]
and if $u\in D(A)=\H_0^1(\Omega)\cap \H^2(\Omega)$ satisfies \eqref{8p2}, then we 
immediately have that 
\[
A_{DG}(u,\chi)=(f,\chi), \quad \ \text{ for all }\  \chi\in V_h^{DG}.
\]
Finally, recalling Lemma \eqref{dglem111}, can write 
\[
\tilde{C}\vertiii{\chi} \le A_{DG}(u_h^{DG},\chi)-A_{DG}(W,\chi)= A_{DG}(u,\chi)-A_{DG}(W,\chi),
\]
	and the rest of the proof follows much in the same way as in Theorems \ref{thm7.1} and \ref{ncthm11}.
\end{proof}

\begin{remark}
Note that we can drive the following $L^2$-error estimates, essentially as a direct consequence of Theorems \ref{thm7.1}, \ref{ncthm11} and \ref{dgerr11}
\[
||u-u_h||_0\le C\,h,\quad ||u-u_h^{CR}||_0\le C\,h, \quad ||u-u_h^{DG}||_{0}\le C\,h,
\]
where the constant $C$ is independent of $h$. These $L^2$-error estimates are however sub-optimal. We nevertheless provide in Section \ref{sec4} numerical evidence that all three numerical methods achieve optimal convergence also in the $L^2-$norm. 
\end{remark}

\section{Numerical results}\label{sec4}
In this section, we present a few computational results that confirm the 
theoretical results advanced in Section~\ref{sec3}. All examples have been 
implemented with the help of the open-source finite element library FEniCS \cite{alnaes_ans15}. 

\subsection{Example 1: Accuracy verification against smooth solutions} 
First we consider problem (\ref{8.1}) defined on the domain $\Omega=(0,1)^d$, where $d=2,3$. The two expressions of the exact solution $u$ are as follows:
\[
\mbox{Case } 1: u= \Pi_{i=1}^{d}(x_i-x_i^2),\qquad 
\mbox{Case } 2: u= \frac{1}{16}\Pi_{i=1}^{d}\sin(\pi x_i).
\]
We choose the values of parameters as follows: $\alpha=0.2$, $\beta=0.1$, $\nu=2$ and $\gamma=0.5$, and the right-hand side 
datum $f$ is manufactured using these closed-form solutions. A sequence of successively refined uniform meshes is constructed 
and the error history (decay of errors measured in the energy and $L^2-$norm as well 
as corresponding convergence rates) for the numerical solutions constructed with CGFEM, NCFEM and DGFEM are reported in what follows. Table \ref{table11-12} presents the convergence results related to Case 1 for 2D and 3D, whereas Table  \ref{table13-14} shows the results pertaining to Case 2. In all tables we can observe that errors in the energy and $L^2-$norms decrease with the mesh size at rates $O(h)$ and $O(h^2)$, respectively. We have used in all simulations a first-order polynomial degree. Other sets of computations performed after modifying the values of the parameter $\delta$ to $3$ and $5$ (not reported here) also show optimal convergence. We can also see that the number of Newton iterations 
required to reach the prescribed tolerance of $10^{-6}$ is at most three. 

\begin{table}[ht!]
\caption{Example 1, case 1. Errors, iteration count, and convergence rates for the numerical solutions $u_h$, $u_h^{CR}$ and $u_h^{DG}$.}
\label{table11-12}
\begin{center}
{\small
\begin{tabular}{| c | c | c | c | c | c | c |}
\hline
\multicolumn{7}{|c|}{Error history in 2D}\\
\hline
\multirow{6}{*}{CGFEM}&\multicolumn{1}{|c|}{mesh}&\multicolumn{1}{|c|} {Newton  it.} &\multicolumn{1}{|c|} {$H^1$-error}&\multicolumn{1}{|c|}{$O(h)$}&\multicolumn{1}{|c|}{$L^2$-error}&\multicolumn{1}{|c|}{$O(h^2)$}\\
\cline{2-7}
&{$4\times 4$}&$3$ &$5.90(-02)$ &$-$ &$5.38(-03)$ & $-$\\
\cline{2-7}
&{$8\times 8$}&$3$ &$3.01(-02)$ &$0.9709$ &$1.42(-03)$ & $1.9217$\\
\cline{2-7}
&{$16\times 16$}&$3$ &$1.51(-02)$ &$0.9952$ &$3.60(-04)$ & $1.9798$\\
\cline{2-7}
&{$32\times 32$}&$3$ &$7.60(-03)$ &$0.9904$ &$9.03(-05)$ & $1.9951$\\
\cline{1-7}
\multirow{5}{*}{NCFEM}
&{$4\times 4$}&$3$ &$4.62(-02)$ &$-$ &$2.32(-03)$ & $-$\\
\cline{2-7}
&{$8\times 8$}&$3$ &$2.35(-02)$ &$0.9752$ &$6.10(-04)$ & $2.1026$\\
\cline{2-7}
&{$16\times 16$}&$3$ &$1.18(-02)$ &$0.9938$ &$1.54(-04)$ & $1.9858$\\
\cline{2-7}
&{$32\times 32$}&$3$ &$5.91(-03)$ &$0.9975$ &$3.88(-05)$ & $1.9888$\\
\cline{1-7}
\multirow{5}{*}{DGFEM}
&{$4\times 4$}&$3$ &$5.83(-02)$ &$-$ &$5.27(-03)$ & $-$\\
\cline{2-7}
&{$8\times 8$}&$3$ &$2.94(-02)$ &$0.9876$ &$1.36(-03)$ & $1.9541$\\
\cline{2-7}
&{$16\times 16$}&$3$ &$1.46(-02)$ &$1.0098$ &$3.40(-04)$ & $2.0000$\\
\cline{2-7}
&{$32\times 32$}&$3$ &$7.25(-03)$ &$1.0099$ &$8.43(-05)$ & $2.0119$\\
\hline
\multicolumn{7}{|c|}{Error history in 3D}\\
\hline
\multirow{6}{*}{CGFEM}&\multicolumn{1}{|c|}{mesh}&\multicolumn{1}{|c|} {Newton  it.} &\multicolumn{1}{|c|} {$H^1$-error}&\multicolumn{1}{|c|}{$O(h)$}&\multicolumn{1}{|c|}{$L^2$-error}&\multicolumn{1}{|c|}{$O(h^2)$}\\
\cline{2-7}
&{$4\times 4\times 4$}&$2$ &$1.63(-02)$ &$-$ &$1.52(-03)$ & $-$\\
\cline{2-7}
&{$8\times 8\times 8$}&$2$ &$8.54(-03)$ &$0.9325$ &$4.22(-04)$ & $1.8487$\\
\cline{2-7}
&{$16\times 16\times 16$}&$2$ &$4.32(-03)$ &$0.9832$ &$1.08(-04)$ & $1.9662$\\
\cline{2-7}
&{$32\times 32\times 32$}&$2$ &$2.16(-03)$ &$1.0000$ &$2.73(-05)$ & $1.9840$\\
\cline{1-7}
\multirow{5}{*}{NCFEM}
&{$4\times 4\times 4$}&$2$ &$1.06(-02)$ &$-$ &$5.42(-04)$ & $-$\\
\cline{2-7}
&{$8\times 8\times 8$}&$2$ &$5.39(-03)$ &$0.9757$ &$1.41(-04)$ & $1.9426$\\
\cline{2-7}
&{$16\times 16\times 16$}&$2$ &$2.70(-03)$ &$0.9973$ &$3.64(-05)$ & $1.9573$\\
\cline{2-7}
&{$32\times 32 \times 32$}&$2$ &$1.35(-03)$ &$1.0000$ &$8.99(-05)$ & $2.0175$\\
\cline{1-7}
\multirow{5}{*}{DGFEM}
&{$4\times 4\times 4$}&$3$ &$1.59(-02)$ &$-$ &$1.44(-03)$ & $-$\\
\cline{2-7}
&{$8\times 8\times 8$}&$3$ &$8.05(-03)$ &$0.9820$ &$3.85(-04)$ & $1.5409$\\
\cline{2-7}
&{$16\times 16\times 16$}&$3$ &$3.94(-03)$ &$1.0308$ &$9.49(-05)$ & $2.0204$\\
\cline{2-7}
&{$32\times 32 \times 32$}&$3$ &$1.93(-03)$ &$1.0296$ &$2.31(-05)$ & $2.0385$\\
\hline
\end{tabular}}
\end{center}
\end{table}

\begin{table}[ht!]
\caption{Example 1, case 2.  Errors, iteration count, and convergence rates for the numerical solutions $u_h$, $u_h^{CR}$ and $u_h^{DG}$.}
\label{table13-14}
\begin{center}
{\small
\begin{tabular}{| c | c | c | c | c | c | c | }
\hline
\multicolumn{7}{|c|}{Error history in 2D}\\
\hline
\multirow{5}{*}{CGFEM}&\multicolumn{1}{|c|}{mesh}&\multicolumn{1}{|c|} {Newton  it.} &\multicolumn{1}{|c|} {$H^1$-error}&\multicolumn{1}{|c|}{$O(h)$}&\multicolumn{1}{|c|}{$L^2$-error}&\multicolumn{1}{|c|}{$O(h^2)$}\\
\cline{2-7}
& $4\times 4$&$3$ &$1.26(-01)$ &$-$ &$1.08(-02)$& $-$\\
\cline{2-7}
&{$8\times 8$}&$3$ &$6.84(-02)$ &$0.8814$ &$3.21(-03)$ & $1.7504$\\
\cline{2-7}
&{$16\times 16$}&$3$ &$3.49(-02)$ &$0.9708$ &$8.45(-04)$ & $1.9256$\\
\cline{2-7}
&{$32\times 32$}&$3$ &$1.75(-02)$ &$0.9959$ &$2.14(-04)$ & $1.9813$\\
\cline{1-7}
\multirow{4}{*}{NCFEM}& $4\times 4$&$3$ &$1.22(-01)$ &$-$ &$7.62(-02)$& $-$\\
\cline{2-7}
&{$8\times 8$}&$3$ &$6.44(-02)$ &$0.9217$ &$2.09(-03)$ & $1.8663$\\
\cline{2-7}
&{$16\times 16$}&$3$ &$3.26(-02)$ &$0.9822$ &$5.38(-04)$ & $1.9578$\\
\cline{2-7}
&{$32\times 32$}&$3$ &$1.63(-02)$ &$0.9912$ &$1.35(-04)$ & $1.9946$\\
\cline{1-7}
\multirow{5}{*}{DGFEM}&{$4\times 4$}&$3$ &$1.23(-01)$ &$-$ &$1.01(-02)$& $-$\\ 
\cline{2-7}
&{$8\times 8$}&$3$ &$6.58(-02)$ &$0.9025$ &$2.99(-03)$ & $1.7561$\\
\cline{2-7}
&{$16\times 16$}&$3$ &$3.34(-02)$ &$0.9782$ &$7.86(-04)$ & $1.9275$\\
\cline{2-7}
&{$32\times 32$}&$3$ &$1.68(-02)$ &$0.9914$ &$1.99(-04)$ & $1.9818$\\
\hline
\multicolumn{7}{|c|}{Error history in 3D}\\
\hline
\multirow{6}{*}{CGFEM}&\multicolumn{1}{|c|}{mesh}&\multicolumn{1}{|c|} {Newton  it.} &\multicolumn{1}{|c|} {$H^1$-error}&\multicolumn{1}{|c|}{$O(h)$}&\multicolumn{1}{|c|}{$L^2$-error}&\multicolumn{1}{|c|}{$O(h^2)$}\\
\cline{2-7}
& $4\times 4 \times 4$&$3$ &$1.07(-01)$ &$-$ &$9.25(-03)$& $-$\\
\cline{2-7}
&{$8\times 8\times 8$}&$3$ &$5.98(-02)$ &$0.7650$ &$2.97(-03)$ & $1.4731$\\
\cline{2-7}
&{$16\times 16\times 16$}&$3$ &$3.08(-02)$ &$0.9325$ &$8.04(-04)$ & $1.8487$\\
\cline{2-7}
&{$32\times 32\times 32$}&$3$ &$1.55(-02)$ &$0.9832$ &$2.05(-04)$ & $1.9662$\\
\cline{1-7}
\multirow{5}{*}{NCFEM}&{$4\times 4\times 4$}&$3$ &$8.79(-02)$ &$-$ &$5.09(-03)$& $-$\\ 
\cline{2-7}
&{$8\times 8\times 8$}&$3$ &$4.54(-02)$ &$0.9159$ &$1.39(-03)$ & $1.7789$\\
\cline{2-7}
&{$16\times 16\times 16$}&$3$ &$2.29(-02)$ &$0.9757$ &$3.56(-04)$ & $1.9426$\\
\cline{2-7}
&{$32\times 32\times 32$}&$3$ &$1.14(-02)$ &$0.9973$ &$8.97(-05)$ & $1.9573$\\
\cline{1-7}
\multirow{5}{*}{DGFEM}&{$4\times 4\times 4$}&$3$ &$1.00(-01)$ &$-$ &$8.03(-03)$& $-$\\ 
\cline{2-7}
&{$8\times 8\times 8$}&$3$ &$5.38(-02)$ &$0.8943$ &$2.51(-03)$ & $1.6777$\\
\cline{2-7}
&{$16\times 16\times 16$}&$3$ &$2.74(-02)$ &$0.9734$ &$6.74(-04)$ & $1.8969$\\
\cline{2-7}
&{$32\times 32\times 32$}&$3$ &$1.37(-02)$ &$1.0000$ &$1.71(-04)$ & $1.9788$\\
\hline
\end{tabular}}
\end{center}
\end{table}

\subsection{Example 2: Stationary wave solution} 
Next we consider (\ref{8.1}) endowed with non-homogeneous Dirichlet boundary conditions. 
The domain is again as in Example 1, and the setup of the problem has been adopted 
from \cite{VJE}, where the exact solution is 
\[
u= 0.5-0.5\tanh(z/(r-\bar{\alpha})), 
\]
with $r= \sqrt{\bar{\alpha}^2+8}$ and $\bar{\alpha}=\alpha\sqrt{2}$.  The values of the 
model parameters are now $\alpha=0.2$, $\beta=1$, $\nu=16$ and $\gamma=0.5$. In  Table \ref{table15-16} we present 
the convergence rates associated with the errors in the energy norm as well as $L^2$-norm for CGFEM, NCFEM and DGFEM. Again we observe optimal convergence in all instances.

\begin{table}[ht!]
\caption{Example 2. Errors, iteration count, 
and convergence rates for the numerical solutions $u_h$, $u_h^{CR}$ and $u_h^{DG}$.}
\label{table15-16}
\begin{center}
{\small
\begin{tabular}{| c | c | c | c | c | c | c | }
\hline
\multicolumn{7}{|c|}{Error history in 2D}\\
\hline
\multirow{6}{*}{CGFEM}&\multicolumn{1}{|c|}{mesh}&\multicolumn{1}{|c|} {Newton  it.} &\multicolumn{1}{|c|} {$H^1$-error}&\multicolumn{1}{|c|}{$O(h)$}&\multicolumn{1}{|c|}{$L^2$-error}&\multicolumn{1}{|c|}{$O(h^2)$}\\
\cline{2-7}
&{$4\times 4$}&$3$ &$1.16(-02)$ &$-$ &$8.99(-04)$ & $-$\\
\cline{2-7}
&{$8\times 8$}&$3$ &$5.83(-03)$ &$0.9926$ &$2.26(-04)$ & $1.9920$\\
\cline{2-7}
&{$16\times 16$}&$3$ &$2.91(-03)$ &$1.0025$ &$5.67(-05)$ & $1.9949$\\
\cline{2-7}
&{$32\times 32$}&$3$ &$1.45(-03)$ &$1.0050$ &$1.41(-05)$ & $2.0077$\\
\cline{1-7}
\multirow{5}{*}{NCFEM}
&{$4\times 4$}&$3$ &$7.96(-03)$ &$-$ &$3.91(-04)$ & $-$\\
\cline{2-7}
&{$8\times 8$}&$3$ &$3.98(-03)$ &$1.0000$ &$9.80(-05)$ & $1.9963$\\
\cline{2-7}
&{$16\times 16$}&$3$ &$1.99(-03)$ &$1.0000$ &$2.45(-05)$ & $2.0000$\\
\cline{2-7}
&{$32\times 32$}&$3$ &$9.96(-04)$ &$0.9986$ &$6.13(-06)$ & $1.9988$\\
\cline{1-7}
\multirow{5}{*}{DGFEM}
&{$4\times 4$}&$3$ &$1.13(-02)$ &$-$ &$8.84(-04)$ & $-$\\
\cline{2-7}
&{$8\times 8$}&$3$ &$5.57(-03)$ &$1.0206$ &$2.19(-04)$ & $2.0131$\\
\cline{2-7}
&{$16\times 16$}&$3$ &$2.76(-03)$ &$1.0130$ &$5.47(-05)$ & $2.0013$\\
\cline{2-7}
&{$32\times 32$}&$3$ &$1.37(-03)$ &$1.0105$ &$1.36(-05)$ & $2.0079$\\
\hline
\multicolumn{7}{|c|}{Error history in 3D}\\
\hline
\multirow{6}{*}{CGFEM}&\multicolumn{1}{|c|}{mesh}&\multicolumn{1}{|c|} {Newton  it.} &\multicolumn{1}{|c|} {$H^1$-error}&\multicolumn{1}{|c|}{$O(h)$}&\multicolumn{1}{|c|}{$L^2$-error}&\multicolumn{1}{|c|}{$O(h^2)$}\\
\cline{2-7}
&{$4\times 4\times 4$}&$3$ &$2.39(-02)$ &$-$ &$1.98(-03)$ & $-$\\
\cline{2-7}
&{$8\times 8\times 8$}&$3$ &$1.19(-02)$ &$1.0060$ &$5.01(-04)$ & $1.9826$\\
\cline{2-7}
&{$16\times 16\times 16$}&$3$ &$5.98(-03)$ &$0.9927$ &$1.25(-04)$ & $2.0029$\\
\cline{2-7}
&{$32\times 32\times 32$}&$3$ &$2.99(-03)$ &$1.0000$ &$3.14(-05)$ & $1.9931$\\
\cline{1-7}
\multirow{5}{*}{NCFEM}
&{$4\times 4\times 4$}&$3$ &$1.35(-02)$ &$-$ &$7.07(-04)$ & $-$\\
\cline{2-7}
&{$8\times 8\times 8$}&$3$ &$6.75(-03)$ &$1.0000$ &$1.77(-04)$ & $1.9980$\\
\cline{2-7}
&{$16\times 16\times 16$}&$3$ &$3.37(-03)$ &$1.0021$ &$4.42(-05)$ & $2.0016$\\
\cline{2-7}
&{$32\times 32 \times 32$}&$3$ &$1.68(-04)$ &$1.0043$ &$1.10(-05)$ & $2.0065$\\
\cline{1-7}
\multirow{5}{*}{DGFEM}
&{$4\times 4\times 4$}&$3$ &$2.30(-02)$ &$-$ &$1.95(-03)$ & $-$\\
\cline{2-7}
&{$8\times 8\times 8$}&$3$ &$1.11(-02)$ &$1.0511$ &$4.84(-04)$ & $2.0104$\\
\cline{2-7}
&{$16\times 16\times 16$}&$3$ &$5.47(-03)$ &$1.0209$ &$1.19(-04)$ & $2.0240$\\
\cline{2-7}
&{$32\times 32\times 32$}&$3$ &$2.70(-03)$ &$1.0186$ &$2.96(-05)$ & $2.0073$\\
\hline
\end{tabular}}
\end{center}
\end{table}

\subsection{Example 3: Application to nerve pulse propagation} 
To conclude this section, and as a qualitative illustration of the differences between a classical bistable 
equation (without advection and with a simplified cubic nonlinearity induced by $\delta = 1$) 
and the generalized Burgers-Huxley equation, we conduct a simple simulation of a  transient 
problem where also an additional ODE (governing the dynamics of a gating variable $v$) 
is considered so that self-sustained patterns are possible (see, e.g., \cite{murray,bini10}). 
The system reads 
\begin{equation}\label{transientBH}
 \partial_t u + \alpha u^\delta \sum_{i=1}^d\partial_i u - \nu \Delta u  - \beta u (1-u^\delta)(u^\delta-\gamma) + v = 0, \qquad \partial_t v = \varepsilon(u-\rho v).\end{equation}
Setting $\delta = 1$ and $\alpha = 0$, one recovers the well-known FitzHugh-Nagumo equations 
$$ \partial_t u - \nu \Delta u - \beta u (1-u)(u-\gamma) + v = 0, \qquad \partial_t v = \varepsilon(u-\rho v).$$ 
We apply a simple backward Euler time discretization with constant time step $\Delta t = 0.2$, after which we recover a discrete formulation 
resembling \eqref{7p1} for the CFEM (and similarly for the other two methods). The domain $\Omega = (0,300)^2$ is discretized into 
a uniform triangular mesh with 25K elements, and the model parameters are taken as 
$\alpha = 0.1, \delta = 1.5, \beta = \nu = 1, \varepsilon = \gamma = 0.01, \rho = 0.05$ (see also \cite{buerger10} for the classical FitzHugh-Nagumo parameters, whereas the modified terms adopt here very mild values). 
For this example we prescribe Neumann boundary conditions for $u$ on $\partial\Omega$. Figure~\ref{fig:ex3} 
depicts three snapshots of the evolution of $u$ (representing the action potential propagation in a 
piece of nerve tissue, cardiac muscle, or any excitable media) for the classical FitzHugh-Nagumo system 
vs. the modified generalized Burgers-Huxley system \eqref{transientBH}, all numerical solutions computed using the DGFEM setting $\gamma = 2$. The differences in spiral dynamics (initiated with a cross-shaped and shifted initial condition for $u$ and $v$) seem to be more sensitive to 
the amount of additional nonlinearity (encoded in $\delta$), rather than to the intensity of the additional advection (modulated by 
$\alpha$). 

\begin{figure}[ht!]
\caption{Example 3. Snapshots at $t = 80, 200, 650$ of $u_h^{DG}$ for the FitzHugh-Nagumo model using $\delta = 1$, $\alpha = 0$ (top panels) and 
for the modified generalized Burgers-Huxley system \eqref{transientBH} with $\delta = 1$, $\alpha = 0.1$ (middle row) and 
with $\delta = 1.5$, $\alpha = 0.1$ (bottom).}\label{fig:ex3}
\begin{center}
\includegraphics[width=0.325\textwidth]{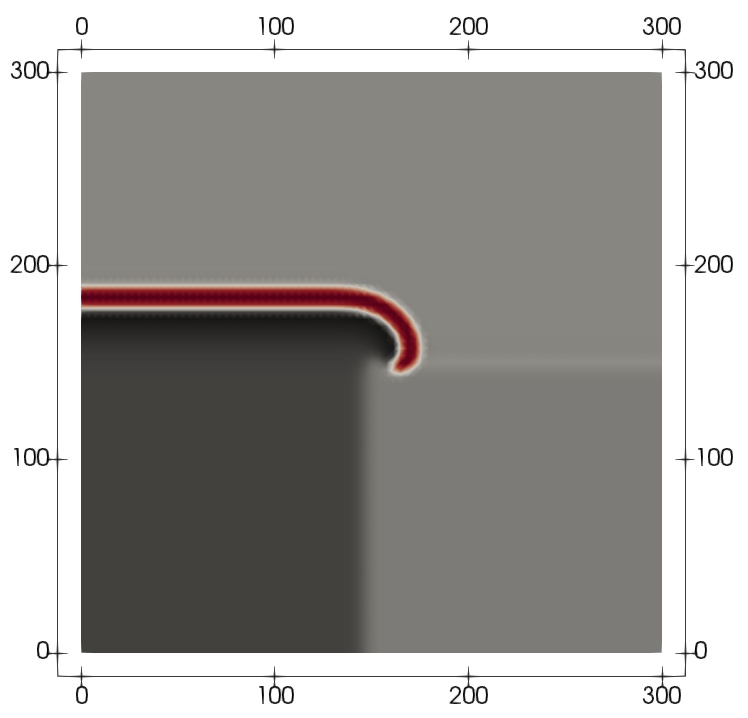}
\includegraphics[width=0.325\textwidth]{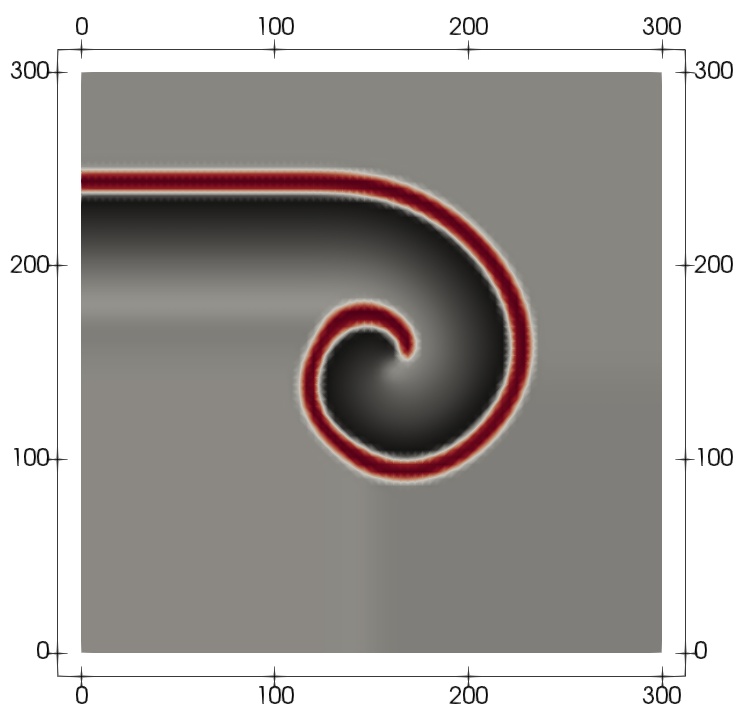}
\includegraphics[width=0.325\textwidth]{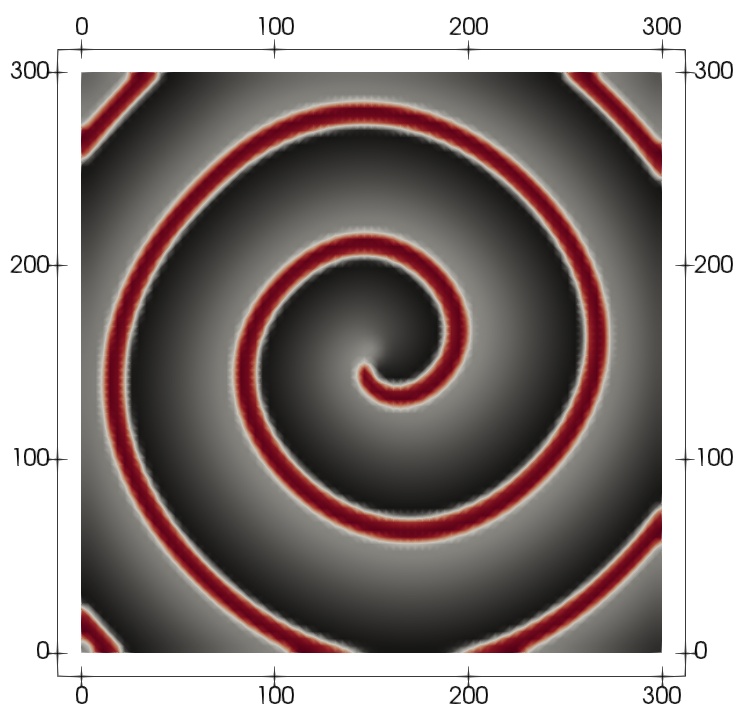}\\
\includegraphics[width=0.325\textwidth]{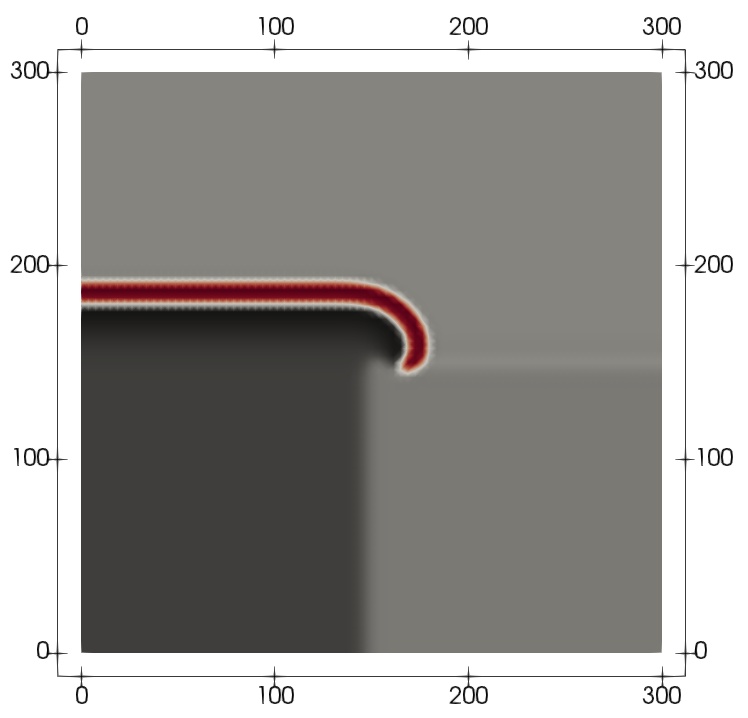}
\includegraphics[width=0.325\textwidth]{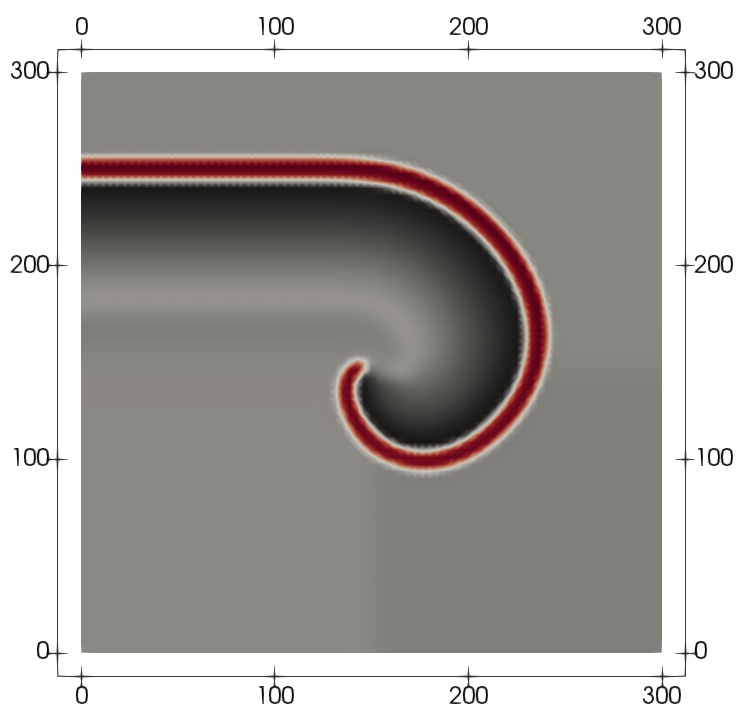}
\includegraphics[width=0.325\textwidth]{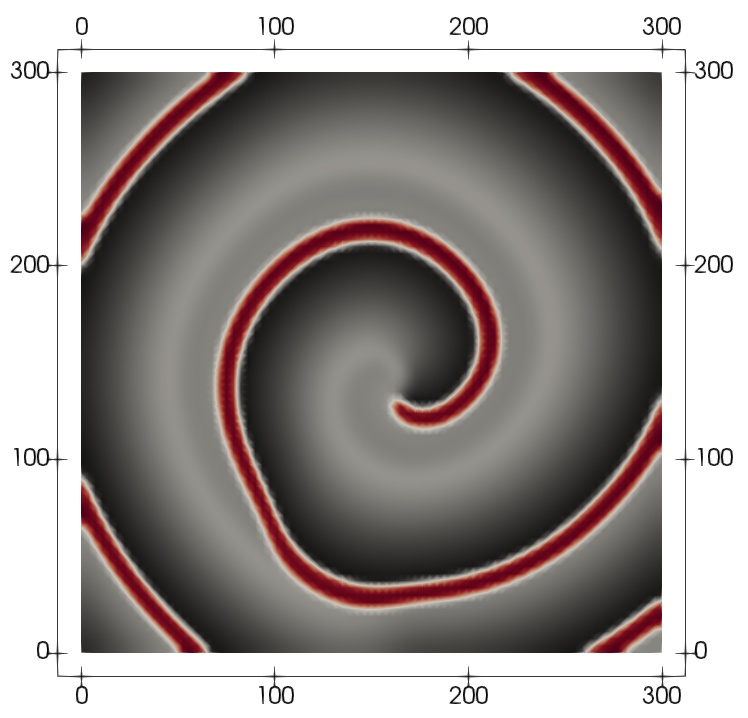}\\
\includegraphics[width=0.325\textwidth]{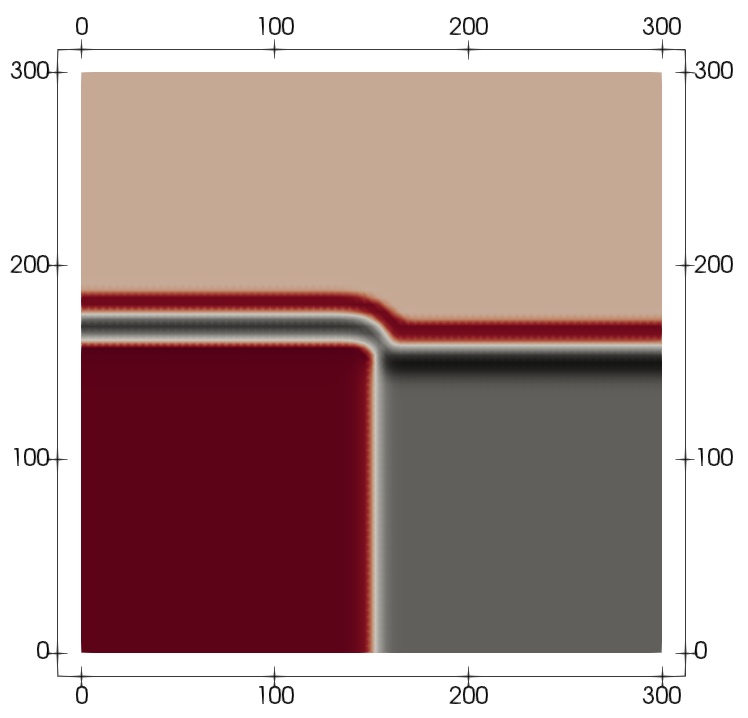}
\includegraphics[width=0.325\textwidth]{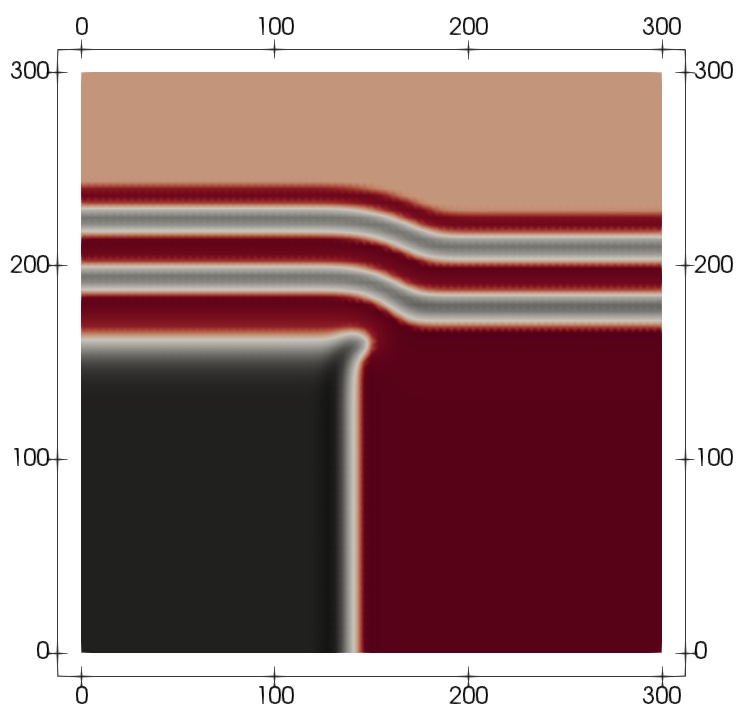}
\includegraphics[width=0.325\textwidth]{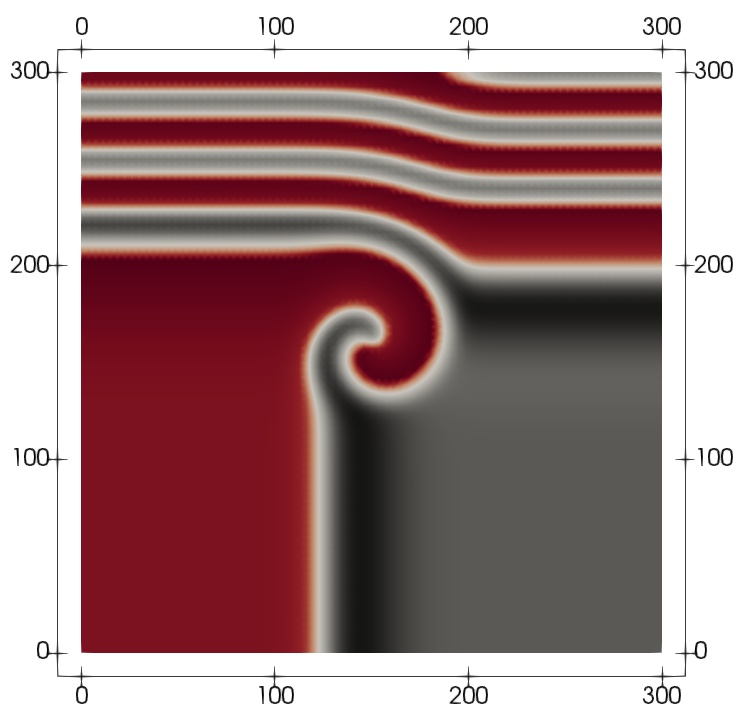}
\end{center}
\end{figure}

\section{Conclusion} 
In this paper we have addressed two main contributions. First, we have proved the well-posedness for the stationary generalized Burgers-Huxley equation. Moreover, we have established a new regularity result that only uses minimal theoretical requirements. 
Secondly, we have introduced three types of finite element approximations (CFEM, NCFEM and DGFEM) for \eqref{8.1}. We have rigorously derived a priori error estimates for all of these discretizations. Finally, computational results are given to validate the theoretical first-order  convergence of the methods. 
As a next step we are extending the theory to cover the transient case, and we will also construct efficient and reliable residual-based a posteriori error estimators and adaptive schemes. We also plan to address the formulation of other conservative discretizations using adequate mixed methods.

\medskip\noindent

\bibliographystyle{siam} 
\bibliography{newReferences}

\end{document}